\documentclass[a4paper, british, oneside]{amsart}

\usepackage[utf8]{inputenc}
\usepackage[T1]{fontenc}

\usepackage[british]{babel}
\usepackage{csquotes}

\usepackage{amsmath,amsthm,amssymb}
\usepackage[margin=2cm]{geometry}
\usepackage[all]{xy}
\usepackage[hidelinks]{hyperref}
\usepackage{microtype}
\usepackage{enumitem}

\renewcommand{\MR}[1]{}

\theoremstyle{definition}
\newtheorem{theorem}{Theorem}[section]
\newtheorem{lemma}[theorem]{Lemma}
\newtheorem{proposition}[theorem]{Proposition}
\newtheorem{corollary}[theorem]{Corollary}
\newtheorem{definition}[theorem]{Definition}

\theoremstyle{remark}
\newtheorem{remark}[theorem]{Remark}

\newtheorem{example}[theorem]{Example}

\DeclareMathOperator{\Th}{Th}
\DeclareMathOperator{\Spec}{Spec}
\DeclareMathOperator{\Frac}{Frac}
\DeclareMathOperator{\Gal}{Gal}

\newcommand{\unspacedrt}[2][]{\mspace{-3mu}\sqrt[#1]{#2}}

\usepackage[backend=biber,style=alphabetic,
  doi=false,isbn=false,
  maxbibnames=10, maxalphanames=4,
  sorting=nyt,
  giveninits=true,
  block = space
  ]{biblatex}

\DeclareSourcemap{
  \maps{
    \map{
      \step[fieldsource=edition,
            match=First,
            replace=1]
      \step[fieldsource=edition,
            match=Second,
            replace=2]
      \step[fieldsource=edition,
            match=Third,
            replace=3]
    }
  }
}

\renewbibmacro*{addendum+pubstate}{%
  \printfield{pubstate}%
  \newunit\newblock
  \printfield{addendum}}

\renewbibmacro{in:}{%
  \ifentrytype{article}{}{\printtext{\bibstring{in}\intitlepunct}}}
 \DeclareFieldFormat
  [article,inbook,incollection,inproceedings,patent,thesis,unpublished]
  {title}{#1\isdot}

\AtEveryBibitem{\clearfield{pagetotal}}
\AtEveryBibitem{\ifentrytype{book}{\clearfield{pages}}{}}
\AtEveryBibitem{\ifentrytype{book}{\clearfield{url}}{}}
\AtEveryBibitem{\ifentrytype{article}{\clearfield{url}}{}}
\AtEveryBibitem{\ifentrytype{incollection}{\clearfield{url}}{}}

\author{Philip Dittmann}
\address{Department of Mathematics, University of Manchester, Manchester M13 9PL, United Kingdom}
\email{philip.dittmann@manchester.ac.uk}

\title{Existential theories of henselian valued fields under a formal smoothness assumption}

\begin{document}

\begin{abstract}
  We study existential theories of henselian valued fields of positive characteristic with parameters from
  a trivially valued subfield.
  Compared to previous work, we relax perfectness and separability assumptions,
  and instead work with the weaker algebraic hypothesis of formal smoothness over the parameter field,
  which we discuss in detail in our setting.
  Assuming a weak consequence of resolution of singularities, which was already used in previous work,
  we obtain an axiomatisation of the existential theory of such a valued field in terms of
  the existential theory of the residue field, both over the same parameter field.
  This result has natural applications to asymptotic theories of completions
  of function fields of curves.
  We work these out in detail for the case of function fields over fairly general pseudo-algebraically closed fields,
  where we obtain decidability of the sets of universal/existential sentences
  holding in all completions or all but finitely many completions, respectively.
\end{abstract}

\maketitle

\section{Introduction}

Hilbert's tenth problem asked for an algorithm to decide, given a polynomial $f \in \mathbb{Z}[X_1, \dotsc, X_n]$,
whether $f$ has a zero in $\mathbb{Z}$.
This was shown to be undecidable in work of Davis, Putnam, Robinson, and Matiyasevich,
but spawned ongoing investigations of variants in a wide variety of other rings and fields \cite{Shlapentokh_book}.
In the terminology of model theory, these variants concern the decidability of the existential theories of rings in certain languages.
Here we study the existential first-order theory of a henselian valued field $(K,v)$.
Prominent examples for us are power series fields $F(\!(\pi)\!)$ over a base field $F$ with their $\pi$-adic (complete discrete) valuation.

The first-order theories of henselian valued fields are extremely well-understood if their residue characteristic is zero:
by celebrated results of Ax--Kochen and Ershov, such theories are completely determined by the theories of their value group and residue field
\cite[Theorem~5.1]{Dries_lectures-model-theory-valued-fields}.
Moving beyond residue characteristic zero, only partial results are known:
Our understanding is essentially complete in the case of finitely ramified henselian valued fields
(which are always of characteristic zero, but positive residue characteristic) \cite{ADJ_finitely-ramified},
and we have good understanding for tame fields \cite{Kuhlmann_tame} (although some questions remain in mixed characteristic tame fields)
and separably tame fields \cite{KuhlmannPal_separably-tame,Anscombe_lambda-fns-separably-tame}.

Our concern in this article are henselian valued fields of positive characteristic, such as the aforementioned power series fields $F(\!(t)\!)$,
which are never (separably) tame.
A very general result in this regard is the following:
\begin{theorem}[{\cite[Corollary~1.2]{AnscombeFehm_existential-equichar}}]
  Let $(K,v)$ and $(L,w)$ be henselian non-trivially valued fields of positive characteristic.
  Then the residue fields $Kv$ and $Lw$ have the same existential theory
  if and only if the valued fields $(K,v)$ and $(L,w)$ have the same existential theory.
\end{theorem}
While this result means that the existential theory of $(K,v)$ is completely determined by the theory of the residue field $Kv$,
and a decidability result for the existential theory of $(K,v)$ relative to the existential theory of $Kv$ follows \cite[Corollary~1.3]{AnscombeFehm_existential-equichar},
it is not entirely satisfactory in that no parameters are allowed.
In the Hilbert's tenth problem formulation, this corresponds to only allowing polynomials $f$ with coefficients in the prime field,
as opposed to in some naturally given subfield of $K$.
In this vein, only the following rather partial result is obtained using the methods of \cite{AnscombeFehm_existential-equichar}.
\begin{theorem}[{\cite[Proposition~4.2]{DittmannFehm_completions}; essentially \cite[Corollary~5.7]{AnscombeFehm_existential-equichar}}]
  Let $(K,v)$ and $(L,w)$ be henselian non-trivially valued fields of positive characteristic
  with the common trivially valued subfield $C$,
  so $C$ can also be seen as a subfield of the residue fields $Kv$ and $Lw$.
  Assume that $C$, $Kv$ and $Lw$ are all perfect.
  Then the existential theories of $Kv$ and $Lw$ with parameters from $C$ agree if and only if
  the existential theories of the valued fields $(K,v)$ and $(L,w)$ with parameters from $C$ agree.
\end{theorem}

Relaxing the perfectness assumption, a description of existential theories over a trivially valued parameter field
was only obtained later, and only under an additional hypothesis (R4),
a weak version of Resolution of Singularities (in fact implied by the conjecture of Local Uniformisation).
\begin{theorem}[{\cite[Corollary~4.16]{ADF_existential}}]\label{thm:adf-existential-transfer}
  Assume (R4).
  Let $(K,v)$ and $(L,w)$ be henselian non-trivially valued fields with the common trivially valued subfield $C$,
  so $C$ can also be seen as a subfield of the residue fields $Kv$ and $Lw$.
  Assume that $Kv/C$ and $Lw/C$ are separable.
  Then the existential theories of $Kv$ and $Lw$ with parameters from $C$ agree if and only if
  the existential theories of the valued fields $(K,v)$ and $(L,w)$ with parameters from $C$ agree.
\end{theorem}
We discuss (R4) in Section~\ref{sec:main-thms}.
Let us note that it was shown in \cite[Remark~4.17]{ADF_existential} that the assumption of (R4) is necessary for the theorem to hold.
See also \cite{DenefSchoutens,Kartas_tamely-ramified} for other work on existential theories of henselian valued fields relying on variants of Resolution of Singularities.

The assumption that $Kv/C$ and $Lw/C$ be separable is necessary in Theorem~\ref{thm:adf-existential-transfer},
as is clear from the following example.
\begin{example}\label{ex:basic}
  Let $C=\mathbb{F}_p(t)$ for some prime $p>2$,
  and consider the fields $K_0 = C(\!(\pi)\!)$, $K_1 = K_0(\unspacedrt[p]{t})$, $K_2 = K_0(\unspacedrt[p]{\pi+t})$, $K_3 = K_0(\unspacedrt[p]{\pi^2+t})$.
  Then $K_0$ is a complete (in particular henselian) discretely valued field with the $\pi$-adic valuation,
  it contains $C$ as a trivially valued subfield, and has residue field $C$.
  We view $K_1$, $K_2$ and $K_3$ as valued fields with the unique extension of the $\pi$-adic valuation.
  It is easy to see that they all have residue field $C(\unspacedrt[p]{t}) = \mathbb{F}_p(\unspacedrt[p]{t})$.
  However, $K_1$ contains a $p$-th root of $t$, while neither $K_2$ nor $K_3$ does;
  therefore the existential theories with parameters from $C$ differ.
  Furthermore, $K_3$ satisfies the existential sentence $\varphi$ given by
  \[ \exists r, s (r \not\in \mathcal{O} \mathrel\land r^2(s^p-t) \in \mathcal{O}) \]
  in the language of valued fields with parameters from $C$
  (take $r = \pi^{-1}, s = \sqrt[p]{\pi^2+t}$),
  while $K_2$ does not satisfy this sentence:
  indeed, a witness $s$ would need to be chosen such that $s^p-t$ is in the maximal ideal,
  so $s = \sqrt[p]{\pi+t} + m$ for some $m$ in the maximal ideal;
  but then $r^2(s^p-t) = r^2(\pi + m^p)$, and $\pi+m^p$ is a uniformiser since $\pi$ is,
  so for $r$ of negative valuation the element $r^2(\pi+m^p)$ also has negative valuation,
  proving that $K_2$ does not satisfy the sentence.
  The choice of sentence is related to Remark~\ref{rem:comparison-URt} below.

  While $\varphi$ is a statement about valued fields since it uses a predicate for the valuation ring both positively
  and negatively,
  we note that in this particular example,
  it can be rewritten as an existential sentence in the language of rings with parameters from $C$
  since by \cite{AnscombeFehm_characterizing-diophantine}
  the valuation ring is uniformly definable by both an existential and a universal formula
  in the class of henselian non-trivially valued fields containing $C$ as a trivially valued subfield with residue field $C(\unspacedrt[p]{t})$:
  For universal definability, this follows from \cite[Theorem~3.18]{AnscombeFehm_characterizing-diophantine},
  since the extension field $C(\unspacedrt[p]{t})/C$ is not large in the sense of \cite[Definition~3.5]{AnscombeFehm_characterizing-diophantine}
  (essentially because of \cite[Proposition~6.10]{AnscombeFehm_characterizing-diophantine}).
  For existential definability, this follows from \cite[Theorem~3.11]{AnscombeFehm_characterizing-diophantine},
  since the extension field $C(\unspacedrt[p]{t})/C$ does not have embedded residue in the sense of \cite[Definition~3.5]{AnscombeFehm_characterizing-diophantine}
  (see \cite[Theorem~4.1.1]{Dittmann_thesis},
  in analogy to the situation of number fields considered in \cite[Corollary~6.13, Corollary~6.14]{ADF_denseness}).
\end{example}

Our goal here is to relax the separability assumption,
by replacing it with the weaker assumption that the valuation ring is \emph{formally smooth} over the parameter field
(in a suitable topological sense, see Definition~\ref{def:fs} and Remark~\ref{rem:fs-topo-or-not}).
This is a notion from commutative algebra studied in \cite[§§~19--22]{EGA-IV-1}, which we consider in Section~\ref{sec:props-formal-smoothness}.
One of our main results can therefore be stated as follows:
\begin{theorem}[{Corollary~\ref{cor:existential-ake}}]
  Assume (R4).
  Let $(K,v)$ and $(L,w)$ be henselian non-trivially valued fields with the common trivially valued subfield $C$.
  Assume that the valuation rings $\mathcal{O}_v$ and $\mathcal{O}_w$ are formally smooth over $C$.
  Then the existential theories of $Kv$ and $Lw$ with parameters from $C$ agree if and only if
  the existential theories of the valued fields $(K,v)$ and $(L,w)$ with parameters from $C$ agree.
\end{theorem}

While the omission of the separability hypothesis compared to Theorem~\ref{thm:adf-existential-transfer} might seem like a modest gain,
it is essential in applications.
In particular, when $C = C_0(t)$ for some base field $C_0$, or more generally $C$ is the function field of some curve over $C_0$,
then any non-principal ultraproduct of the completions of $C$ with respect to the discrete valuations trivial on $C_0$
will contain $C$ as a trivially valued subfield;
however, the residue field of the ultraproduct will almost never be separable over $C$:
for instance, if $C_0$ is perfect, then so is the residue field of the ultraproduct, and so it cannot be separable over the imperfect field $C$.
Understanding such ultraproducts of completions of $C_0(t)$, or more generally function fields of curves, i.e.\ finite extensions of $C_0(t)$,
is essential for analysing the set of sentences holding in all but finitely many such completions.

The second half of this article, see Section~\ref{sec:function-fields-pac}, is dedicated to such applications to completions of function fields.
This is in analogy to Ax's result of decidability of the sets of sentences holding in all or all but finitely many $\mathbb{Q}_p$,
see \cite[Section~11]{Ax_elementary-theory-finite-fields} and also \cite[Corollary~1.2]{DittmannFehm_completions}
for a generalisation to number fields.
Results of this kind are very natural in the light of the local-global philosophy in arithmetic geometry.
First steps in the function field setting were undertaken in \cite{DittmannFehm_completions},
albeit restricted to function fields over finite fields.
We can now understand this situation in much greater generality.
In particular we obtain the following results for universal/existential sentences, i.e.\ boolean combinations of existential sentences,
over a large class of pseudo-algebraically closed (PAC) fields
(see \cite[Chapter~11]{FriedJarden} and later chapters for an extensive treatment of PAC fields).

\begin{theorem}[Remark~\ref{rem:pac-aa-completions-axiomatisable},
  Corollary~\ref{cor:pac-aa-completions-decidable}, Proposition~\ref{prop:all-completions-decidable}]
  \label{thm:intro-pac-completions}
  Let $K$ be a PAC field of characteristic $p>0$ whose absolute Galois group is a free profinite group
  (e.g.\ $K$ is separably closed or pseudofinite),
  and let $F/K$ be a finitely generated separable field extension of transcendence degree $1$.
  Consider the completions $\widehat{F}^v$ with respect to valuations $v$ on $F$ with value group $\mathbb{Z}$
  which are trivial on $K$.
  \begin{enumerate}
  \item The set of universal/existential $\mathcal{L}_{\mathrm{ring}}(F)$-sentences holding in all but finitely many completions of $F$
    admits an explicit axiomatisation.
  \item Assume (R4). Then the same holds for the universal/existential $\mathcal{L}_{\mathrm{val}}(F)$-sentences.
  \end{enumerate}
  Assume now that $K$ and $F$ are endowed with a sensible indexing by natural numbers,
  so that it makes sense to speak of decidability of sets of $\mathcal{L}_{\mathrm{val}}(F)$-sentences.
  (Such indexings exist for instance if $K$ is the separably closed field $\mathbb{F}_p(t_1, \dotsc, t_e)^{\mathrm{sep}}$
  and in many other cases, see Remark~\ref{rem:constructing-examples-pac};
  the precise hypothesis is that $K$ needs to be a computable field with splitting algorithm and decidable $p$-independence relation.)
  \begin{enumerate}
  \item The set of universal/existential $\mathcal{L}_{\mathrm{ring}}(F)$-sentences holding in all but finitely many completions of $F$
    is decidable.
  \item Assume (R4). Then the set of universal/existential $\mathcal{L}_{\mathrm{val}}(F)$-sentences holding in all but finitely many completions of $F$,
    as well as the set of such sentences holding in all completions of $F$, are both decidable.
  \end{enumerate}
\end{theorem}

The proof of the theorem works by reducing the truth of universal/existential sentences in all but finitely many completions
to the truth of universal/existential $\mathcal{L}_{\mathrm{ring}}(F)$-sentences in all but finitely many residue fields
(Proposition~\ref{prop:theories-aa-completions-to-res-flds}).
Under (R4), this works in complete generality, without restriction of the base field,
crucially relying on our results for formally smooth valuation rings.
Without (R4), we must rely on a weaker ad hoc result \cite[Proposition~5.1]{DittmannFehm_completions}
restricted to PAC base fields.

Our second crucial ingredient to Theorem~\ref{thm:intro-pac-completions} is therefore an understanding of which $\mathcal{L}_{\mathrm{ring}}(F)$-sentences hold
in almost all residue fields $Fv$ of $F/K$.
For a finite base field $K$, this theory of almost all residue fields is simply the theory of pseudofinite fields containing $F$,
see \cite[Lemma~4.5]{DittmannFehm_completions} (summarising results of \cite[Chapter~20]{FriedJarden}).
This beautiful result relies on a close analysis of the theories of pseudofinite fields (following Ax) and the Chebotarev Density Theorem.
For the PAC base fields that interest us, we obtain an analogous but more complicated result.
Here we treat arbitrary sentences (not merely universal/existential ones), and the results and techniques may be of independent interest.
\begin{theorem}[Theorem~\ref{thm:almost-all-res-pac}, Proposition~\ref{prop:almost-all-res-fields-decidable}]\label{thm:intro-pac-aa-res}
  Let $K$ be a PAC field of characteristic $p>0$ with free absolute Galois group,
  and let $F/K$ be a finitely generated separable field extension of transcendence degree $1$.
  We consider the residue fields $Fv$ of $F$ with respect to valuations $v$ on $F$ with value group $\mathbb{Z}$
  which are trivial on $K$. All of these are finite extensions of $K$.
  We take $Fv$ to be an $\mathcal{L}_{\mathrm{ring}}(F)$-structure by letting the constant symbol for $a \in F$ stand for its residue in $Fv$ if $v(a) \geq 0$,
  and for $0$ otherwise (which only happens for finitely many $v$).
  Let $T$ be the set of $\mathcal{L}_{\mathrm{ring}}(F)$-sentences that hold in all but finitely many $Fv$.
  \begin{enumerate}
  \item The set of such $\mathcal{L}_{\mathrm{ring}}(F)$-sentences admits an explicit axiomatisation.
  \item Assume that $K$ and $F$ are endowed with a sensible indexing by natural numbers (cf.\ Theorem~\ref{thm:intro-pac-completions}).
    Then this set of $\mathcal{L}_{\mathrm{ring}}(F)$-sentences is decidable.
  \end{enumerate}
\end{theorem}
This result requires a perhaps surprisingly intricate proof.
It relies on the one hand on the model theory of PAC fields
(as presented in \cite[§~5]{Chatzidakis_properties-forking-omega-free-pac}, after Cherlin--van den Dries--Macintyre),
but on the other hand also needs extensions of techniques regarding the algebra of PAC fields (as detailed in \cite{FriedJarden}),
see for instance Lemma~\ref{lem:separable-half} and Lemma~\ref{lem:inseparablisation} below.
The case of imperfect base fields $K$ causes particular challenges.

\section{Properties of formal smoothness}
\label{sec:props-formal-smoothness}

In this section, we investigate the condition of formal smoothness of a valuation ring over a subfield.
The setting is as follows.
We fix a prime $p>0$ and a field $C$ of characteristic $p$.
We work with valuation rings $\mathcal{O}$ containing $C$.
This corresponds to a valued field $(K,v)$, with $K$ the fraction field of $\mathcal{O}$ and $v$ the corresponding Krull valuation,
such that $K$ contains $C$ as a trivially valued subfield.
We use \cite{EGA-IV-1} as our main reference, but see also \cite[Tag~07EA]{StacksProject}.

The definition of formal smoothness is through the following lifting property;
we already note, however, that we will generally not work with the definition as such,
preferring instead to rely on characterisations in \cite[Chapitre~0]{EGA-IV-1}.
\begin{definition}[{\cite[Définition~19.3.1]{EGA-IV-1}}]\label{def:fs}
  Let $\mathcal{O} \supseteq C$ be a valuation ring with maximal ideal $\mathfrak{m}$.
  Then $\mathcal{O}$ is \emph{formally smooth over $C$} (equivalently, the ring homomorphism $C \to \mathcal{O}$ is formally smooth) if,
  for every $C$-algebra $B$ and every nilpotent ideal $I$ of $B$,
  every $C$-algebra homomorphism $\mathcal{O} \to B/I$
  which is continuous with respect to the $\mathfrak{m}$-adic topology on $\mathcal{O}$ and the discrete topology on $B/I$
  factors through a $C$-algebra morphism $\mathcal{O} \to B$ which is continuous
  (again with respect to the $\mathfrak{m}$-adic topology on $\mathcal{O}$ and the discrete topology on $B$).
\end{definition}
Here the $\mathfrak{m}$-adic topology on $\mathcal{O}$ is the ring topology for which the powers $\mathfrak{m}^n$, $n \in \mathbb{N}$,
form a basis of open neighbourhoods of $0$.
Note that this topology is exactly the valuation topology if $\mathcal{O}$ is a discrete valuation ring (DVR)
and the discrete topology if $\mathcal{O}$ is a field.
In all other cases, this topology is not Hausdorff, since $\bigcap_{n \geq 0} \mathfrak{m}^n \supsetneq \{ 0 \}$.
Continuity of a homomorphism $\mathcal{O} \to B/I$, or $\mathcal{O} \to B$,
where the codomain is endowed with the discrete topology,
precisely means that some power of $\mathfrak{m}$ is contained in the kernel.
\begin{remark}\label{rem:fs-topo-or-not}
  Some care needs to be taken with regards to the terminology defined above.
  To avoid ambiguity, one may prefer to call $\mathcal{O}$ \emph{formally smooth over $C$ for the $\mathfrak{m}$-adic topology on $\mathcal{O}$ and the discrete topology on $C$}
  if the property above is satisfied,
  and use the discrete topology on $\mathcal{O}$ instead when no topology is mentioned.\footnote{%
    The topology on $C$ essentially plays no role for formal smoothness,
    see \cite[Proposition~19.3.8]{EGA-IV-1}, \cite[Tag~07CE]{StacksProject}.
    }
  This is a common convention in the literature.
  However, since we will always use the $\mathfrak{m}$-adic topology on $\mathcal{O}$ when considering formal smoothness,
  we prefer to suppress mentioning the topology every time.
\end{remark}
Given a valuation ring $\mathcal{O}$ with maximal ideal $\mathfrak{m}$,
it is useful to consider its completion\footnote{``séparé complété'' in the terminology of Bourbaki, as used in \cite{EGA-IV-1}}
$\widehat{\mathcal{O}}$, which we always take with respect to the $\mathfrak{m}$-adic topology,
i.e.\ $\widehat{\mathcal{O}} = \varprojlim_{n \geq 0} \mathcal{O}/\mathfrak{m}^n$.
This is not to be confused with the completion of $\mathcal{O}$ with respect to the valuation topology (in the sense of the completion of a uniform space).
The completion $\widehat{\mathcal{O}}$ admits the following description:
If $\mathfrak{m}$ is either the zero ideal (in which case $\mathcal{O}$ is a field) or is not a principal ideal (i.e.\ the valuation does not have a uniformiser),
we have $\mathfrak{m}^n = \mathfrak{m}$ for all $n \geq 1$, and so $\widehat{\mathcal{O}} = \mathcal{O}/\mathfrak{m}$ is simply the residue field of $\mathcal{O}$.
Otherwise, $\mathfrak{p} = \bigcap_{n \geq 0} \mathfrak{m}^n$ is a non-maximal prime ideal of $\mathcal{O}$, $\mathcal{O}/\mathfrak{p}$ is a DVR with maximal ideal $\mathfrak{m}/\mathfrak{p}$,
and $\widehat{\mathcal{O}} = \widehat{\mathcal{O}/\mathfrak{p}}$ is the completion of the DVR $\mathcal{O}/\mathfrak{p}$ in the usual sense.
In particular, we see that $\widehat{\mathcal{O}}$ is either a complete DVR or a field in all cases,
and the natural map $\mathcal{O} \to \widehat{\mathcal{O}}$ is usually not injective.

Formal smoothness can be checked on the completion:
\begin{lemma}\label{lem:fs-completion}
  Let $\mathcal{O} \supseteq C$ be a valuation ring.
  Then $\mathcal{O}$ is formally smooth over $C$ if and only if $\widehat{\mathcal{O}}$ is.
\end{lemma}
\begin{proof}
  See \cite[Chapitre~0, Proposition~19.3.6]{EGA-IV-1}.
\end{proof}

Although this is of secondary importance for us in developing the theory, we note the following consequence:
\begin{lemma}\label{lem:fs-henselisation}
  Let $\mathcal{O} \supseteq C$ be a valuation ring.
  Then $\mathcal{O}$ is formally smooth over $C$ if and only if its henselisation $\mathcal{O}^h$ is.
\end{lemma}
\begin{proof}
  By \cite[Théorème~18.6.6~(iv)]{EGA-IV-4}, the completions $\widehat{\mathcal{O}^h}$ and $\widehat{\mathcal{O}}$
  are isomorphic as $C$-algebras,
  so Lemma~\ref{lem:fs-completion} yields the claim.
\end{proof}

We now develop some criteria for formal smoothness in terms of Kähler differentials.
Recall that when $R, S$ are rings with a homomorphism $S \to R$,
the $R$-module $\Omega_{R/S}$ of Kähler differentials of $R$ over $S$
is the $R$-module generated by elements $dr$, $r \in R$,
modulo the relations $d(a+b) = da + db$ (additivity), $d(a \cdot b) = a \cdot db + b \cdot da$ (Leibniz rule),
and $dr = 0$ whenever $r$ is in the image of $S$.
The by far most important case for us is the module of absolute Kähler differentials $\Omega_R$,
which is the special case of the above for $S = \mathbb{Z}$:
since the Leibniz rule already forces $d(1) = 0$, this means we only impose the
additivity and Leibniz rule relations on the generators $dr$.
If $R$ is of characteristic $p$,
we can also identify $\Omega_R$ with $\Omega_{R/\mathbb{F}_p}$ or $\Omega_{R/R^p}$,
since the Leibniz rule forces $d(r^p) = 0$ for $r \in R$.
\begin{remark}\label{rem:differentials-completion}
  For a valuation ring $\mathcal{O}$ with maximal ideal $\mathfrak{m}$, the differential module $\Omega_{\mathcal{O}}$ is chiefly interesting to us
  through the tensor product $\Omega_{\mathcal{O}} \otimes_{\mathcal{O}} \mathcal{O}/\mathfrak{m}$.
  As far as this $\mathcal{O}/\mathfrak{m}$-vector space is concerned, $\mathcal{O}$ may be freely replaced by its completion:
  indeed, $\mathcal{O} \to \mathcal{O}/\mathfrak{m}^2$ induces an isomorphism $\Omega_{\mathcal{O}} \otimes_{\mathcal{O}} \mathcal{O}/\mathfrak{m} \to \Omega_{(\mathcal{O}/\mathfrak{m}^2)} \otimes_{\mathcal{O}/\mathfrak{m}^2} \mathcal{O}/\mathfrak{m}$ by \cite[Tag~02HQ]{StacksProject},
  and $\mathcal{O}/\mathfrak{m}$, $\mathcal{O}/\mathfrak{m}^2$ remain the same under passage to the completion.
\end{remark}

\begin{lemma}\label{lem:fs-differential-criterion}
  Let $\mathcal{O} \supseteq C$ be a valuation ring with maximal ideal $\mathfrak{m}$.
  Then $\mathcal{O}$ is formally smooth over $C$ if and only if the map
  \[ \Omega_{C} \otimes_C \mathcal{O}/\mathfrak{m} \to \Omega_{\mathcal{O}} \otimes_{\mathcal{O}} \mathcal{O}/\mathfrak{m} ,\]
  induced by the natural map $\Omega_{C} \otimes_C \mathcal{O} \to \Omega_{\mathcal{O}}, dc \otimes x \mapsto x \cdot dc$,
  is injective.
\end{lemma}
\begin{proof}
  See \cite[Chapitre~0, Théorème~22.2.2]{EGA-IV-1} and (a) of the following Remarque.
  
  With other references like \cite[Tag 07EL]{StacksProject}, one only immediately gets this for $\mathcal{O}$ a DVR or a field.
  However, by Remark~\ref{rem:differentials-completion} we may always replace $\mathcal{O}$ by its completion,
  and this is a DVR or a field.
\end{proof}

For later use, we note the following exact sequence relating the modules of differentials $\Omega_{\mathcal{O}}$ and $\Omega_{\mathcal{O}/\mathfrak{m}}$.
\begin{lemma}\label{lem:conormal-sequence}
  Let $\mathcal{O}$ be a valuation ring of characteristic $p$ with maximal ideal $\mathfrak{m}$.
  We have an exact sequence of $\mathcal{O}/\mathfrak{m}$-vector spaces
  \[ 0 \to \mathfrak{m}/\mathfrak{m}^2 \to \Omega_{\mathcal{O}} \otimes_{\mathcal{O}} \mathcal{O}/\mathfrak{m} \to \Omega_{\mathcal{O}/\mathfrak{m}} \to 0 ,\]
  where an element $x + \mathfrak{m}^2 \in \mathfrak{m}/\mathfrak{m}^2$ is mapped to $dx \otimes 1 \in \Omega_\mathcal{O} \otimes_{\mathcal{O}} \mathcal{O}/\mathfrak{m}$,
  and an element $da \otimes b \in \Omega_{\mathcal{O}} \otimes_{\mathcal{O}} \mathcal{O}/\mathfrak{m}$ is mapped to $b \cdot d(a + \mathfrak{m}) \in \Omega_{\mathcal{O}/\mathfrak{m}}$.
\end{lemma}
\begin{proof}
  This is a special case of the conormal sequence \cite[Proposition~16.3]{Eisenbud_comm-algebra} associated
  to the ring homomorphisms $\mathbb{F}_p \to \mathcal{O} \to \mathcal{O}/\mathfrak{m}$.
  In general this is only right exact, but here left exactness is ensured by \cite[Corollary~16.13]{Eisenbud_comm-algebra},
  since we are working over the perfect field $\mathbb{F}_p$.
\end{proof}

For a field $K$ of characteristic $p$, Kähler differentials characterise $p$-independence.
Let $K_0 \subseteq K$ be a subfield.
Recall that a family $(x_i)_{i \in I}$ of elements of $K$ is relatively $p$-independent over $K_0$ if $K^pK_0(\{ x_i \colon i \in I_0 \}) \subsetneq K^pK_0(\{ x_i \colon i \in I \})$ for all $I_0 \subsetneq I$.
It is a relative $p$-basis of $K/K_0$ if additionally $K^p K_0(\{ x_i \colon i \in I \}) = K$, or equivalently, if it is a maximal relatively $p$-independent family
\cite[§~4]{MacLane_sep-trans-bases}.
Note that relative $p$-independence over $K_0$ is the same as relative $p$-independence over $K^p K_0$;
some sources such as \cite[Chapitre~V, §13]{Bourbaki_algebre-4567} restrict the definitions by requiring that $K$ be purely inseparable of height at most $1$ over $K_0$,
but this can always be achieved by replacing $K_0$ by $K^p K_0$.
The family $(x_i)_{i \in I}$ is relatively $p$-independent (a relative $p$-basis) over $K_0$ if and only if
the elements $dx_i \in \Omega_{K/K_0}$ are $K$-linearly independent (form a $K$-vector space basis of $\Omega_{K/K_0}$, respectively)
\cite[Chapitre~V, §~13, No~2, Théorème~1]{Bourbaki_algebre-4567}.
The by far most important case for us is for the subfield $K_0 = \mathbb{F}_p$ (or equivalently $K_0 = K^p$).
Here we speak of absolute $p$-independence and absolute $p$-bases, or simply $p$-independence and $p$-bases without further qualification.
By the above, these properties are characterised by the absolute differential module $\Omega_K$.

If a family $t_0, \dotsc, t_n \in K$ is $p$-dependent, there exists an index $i_0$ such that $t_{i_0} \in K^p(\{ t_i \colon 0 \leq i \leq n, i \neq i_0 \})$.
After reïndexing, we may assume $i_0 = 0$, so $t_0 \in K^p(t_1, \dotsc, t_n)$.
This means that there exist coefficients $x_\alpha \in K$ indexed by multi-indices $\alpha = (\alpha_1, \dotsc, \alpha_n)$ with entries $0 \leq \alpha_i < p$
such that $t_0 = \sum_\alpha x_\alpha^p t^\alpha$, where we write $t^\alpha = t_1^{\alpha_1} \dotsm t_n^{\alpha_n}$,
since the $t^\alpha$ generate the $K^p$-vector space $K^p(t_1, \dotsc, t_n)$.

With a view towards reinterpreting the injectivity statement of Lemma~\ref{lem:fs-differential-criterion},
we can now give a somewhat complicated but elementary characterisation of linear independence in $\Omega_{\mathcal{O}} \otimes_{\mathcal{O}} \mathcal{O}/\mathfrak{m}$.
In the following, as is common practice, we generally omit the index in the tensor product symbol $\otimes$
when there is no risk of ambiguity for the base ring.
\begin{lemma}\label{lem:diff-indep-ad-hoc-crit}
  Let $\mathcal{O}$ be a valuation ring of characteristic $p$ with maximal ideal $\mathfrak{m}$ and $T \subseteq \mathcal{O}$ a subset.
  The following are equivalent:
  \begin{enumerate}
  \item The elements $dt \otimes 1 \in \Omega_{\mathcal{O}} \otimes \mathcal{O}/\mathfrak{m}$, $t \in T$, are $\mathcal{O}/\mathfrak{m}$-linearly independent.
  \item For every list of finitely many distinct elements $\gamma, t_1, \dotsc, t_n \in T$, the following holds:
    For all choices of elements $x_\alpha \in \mathcal{O}$, indexed by multi-indices $\alpha = (\alpha_1, \dotsc, \alpha_n)$ with entries $0 \leq \alpha_i < p$,
    and the resulting element $\hat{x} = \gamma - \sum_\alpha x_\alpha^p t^\alpha$,
    we either have $\hat{x} \not\in \mathfrak{m}$,
    or $\hat{x}$ generates $\mathfrak{m}$
    and the residues of the elements $t_i$ in $\mathcal{O}/\mathfrak{m}$ are $p$-independent.
  \end{enumerate}
\end{lemma}
\begin{proof}
  Throughout, we use the exact sequence
  $0 \to \mathfrak{m}/\mathfrak{m}^2 \to \Omega_{\mathcal{O}} \otimes \mathcal{O}/\mathfrak{m} \to \Omega_{\mathcal{O}/\mathfrak{m}} \to 0$
  from Lemma~\ref{lem:conormal-sequence}.
  Let us suppose the second statement holds.
  Consider finitely many distinct elements $t_0, \dotsc, t_n \in T$.
  If their residues $\overline{t_i}$ in $\mathcal{O}/\mathfrak{m}$ are $p$-independent,
  so the $d\overline{t_i}$ are linearly independent in $\Omega_{\mathcal{O}/\mathfrak{m}}$,
  then certainly the $dt_i \otimes 1$ are linearly independent in $\Omega_{\mathcal{O}} \otimes \mathcal{O}/\mathfrak{m}$.
  Let us suppose that this is not the case.
  Without loss of generality, $\overline{t_0}$ lies in the field $(\mathcal{O}/\mathfrak{m})^p[\{ \overline{t_i} \colon i \geq 1 \}]$.
  We set $\gamma := t_0$ and find elements $x_\alpha \in \mathcal{O}$ indexed by multi-indices $\alpha = (\alpha_1, \dotsc, \alpha_n)$ with entries $0 \leq \alpha_i < p$
  such that $\overline \gamma = \sum_\alpha \overline{x_\alpha}^p \overline{t}^\alpha$ (where we write $\overline{x_\alpha} \in \mathcal{O}/\mathfrak{m}$ for the residue of $x_\alpha$).
  The element $\hat{x} = \gamma - \sum_\alpha x_\alpha^p t^\alpha$ therefore lies in $\mathfrak{m}$,
  so by our assumption it must generate $\mathfrak{m}$,
  and furthermore the $t_1, \dotsc, t_n$ have $p$-independent residues in $\mathcal{O}/\mathfrak{m}$.
  By the exact sequence, it follows that the elements
  $d\hat{x} \otimes 1, dt_1 \otimes 1, \dotsc, dt_n \otimes 1$ are linearly independent in $\Omega_{\mathcal{O}} \otimes \mathcal{O}/\mathfrak{m}$;
  therefore the same holds on replacing $d\hat{x}$ by $d\gamma = dt_0$,
  since $d\gamma - d\hat{x}$ is a linear combination of $dt_1, \dotsc, dt_n$
  (apply the Leibniz rule to the definition of $\hat{x}$).
  Since our choice of elements $t_i$ was arbitrary,
  this proves the first statement.

  Let us suppose the second statement is not true,
  so there exist distinct elements $\gamma, t_1, \dotsc, t_n \in T$
  and elements $x_\alpha \in \mathcal{O}$ violating the condition.
  Suppose $\hat{x}$ is in $\mathfrak{m}$ but does not generate it, so that $\hat{x} \in \mathfrak{m}^2$.
  Writing $\hat{x}$ as a sum of non-trivial products of elements of $\mathfrak{m}$ and applying the Leibniz rule,
  we see that $d\hat{x} \otimes 1 \in \Omega_{\mathcal{O}} \otimes \mathcal{O}/\mathfrak{m}$ vanishes,
  so $d\gamma \otimes 1 = \sum_\alpha d(x_\alpha^p t^\alpha) \otimes 1$.
  Expanding the right-hand side using the Leibniz rule and the fact $d(x_\alpha^p) = p \cdot x_\alpha^{p-1} dx_\alpha = 0$,
  we obtain an $\mathcal{O}/\mathfrak{m}$-linear combination of the $dt_i \otimes 1$,
  contradicting the first statement.
  Let us suppose on the other hand that $\hat{x}$ generates $\mathfrak{m}$,
  but the residues of the elements $t_i$ in $\mathcal{O}/\mathfrak{m}$ are $p$-dependent.
  If the elements $d\gamma \otimes 1, dt_1 \otimes 1, \dotsc, dt_n \otimes 1$ are linearly independent,
  then so are the elements $d\hat{x} \otimes 1, dt_1 \otimes 1, \dotsc, dt_n \otimes 1$.
  Since the first of these lies in the image of the one-dimensional $\mathcal{O}/\mathfrak{m}$-vector space $\mathfrak{m}/\mathfrak{m}^2$,
  it follows from the exact sequence that the $d\overline{t}_i$ must be linearly independent in $\Omega_{\mathcal{O}/\mathfrak{m}}$,
  in contradiction to our assumption of $p$-dependence.
  Therefore the $d\gamma \otimes 1, dt_1 \otimes 1, \dotsc, dt_n \otimes 1$ must have been linearly dependent after all,
  showing that in all cases the first statement is not true.
\end{proof}

Let us now deduce a first-order characterisation of formal smoothness.
We work in the language of rings $\mathcal{L}_{\mathrm{ring}} = \{+,\mathrel{-},\cdot,0,1\}$.
As usual, $\mathcal{L}_{\mathrm{ring}}(C)$ is the language of rings expanded by constant symbols for elements of $C$,
so that rings extending $C$ are naturally $\mathcal{L}_{\mathrm{ring}}(C)$-structures.
Later we will use the one-sorted language of valued fields $\mathcal{L}_{\mathrm{val}} = \mathcal{L}_{\mathrm{ring}} \cup \{ \mathcal{O} \}$
with a unary predicate for the valuation ring.
\begin{corollary}\label{cor:fo-independent-differentials}
  Let $n > 0$.
  There exists an $\mathcal{L}_{\mathrm{ring}}$-formula in $n$ free variables which
  holds in a valuation ring $\mathcal{O}$ of characteristic $p$ with maximal ideal $\mathfrak{m}$ for $n$ given elements $t_1, \dotsc, t_n \in \mathcal{O}$
  if and only if the $dt_i \otimes 1 \in \Omega_{\mathcal{O}} \otimes \mathcal{O}/\mathfrak{m}$ are $\mathcal{O}/\mathfrak{m}$-linearly independent.
\end{corollary}
\begin{proof}
  Immediate from the lemma:
  Take a formula expressing that the $t_i$ are distinct and that
  the second condition of the lemma holds
  for every permutation of the $t_i$.
\end{proof}

\begin{corollary}\label{cor:fo-fs}
  There is a set of $\mathcal{L}_{\mathrm{ring}}(C)$-sentences such that a valuation ring $\mathcal{O} \supseteq C$
  satisfies all of them if and only if it is formally smooth over $C$.
\end{corollary}
\begin{proof}
  Fix a $p$-basis $T \subseteq C$.
  By the previous corollary we can find a set of sentences with parameters from $C$
  expressing that the elements $dt \otimes 1 \in \Omega_{\mathcal{O}} \otimes \mathcal{O}/\mathfrak{m}$, $t \in T$, are linearly independent.
  Since the elements $dt$ form a basis of the $C$-vector space $\Omega_{C}$,
  and so the $dt \otimes 1$ form a basis of the $\mathcal{O}/\mathfrak{m}$-vector space $\Omega_{C} \otimes \mathcal{O}/\mathfrak{m}$,
  the chosen set of sentences is as desired by Lemma~\ref{lem:fs-differential-criterion}.
\end{proof}

\begin{remark}\label{rem:fo-fs-computable}
  In a context where this makes sense,
  we can ask that there be a \emph{computably enumerable} set of $\mathcal{L}_{\mathrm{ring}}(C)$-sentences as above.
  So let us suppose that $C$ is countable and endowed with a bijection from a decidable subset $I \subseteq \mathbb{N}$.
  We obtain an injection of the set of $\mathcal{L}_{\mathrm{ring}}(C)$-formulas into $\mathbb{N}$ via a standard Gödel coding,
  so it makes sense to consider notions of decidability for sets of formulas.
  It is clear from the proof of Corollary~\ref{cor:fo-fs} that there is
  a computably enumerable set of sentences characterising formal smoothness over $C$ as soon as $C$ has a computably enumerable $p$-basis,
  i.e.\ a $p$-basis whose preimage under the bijection $I \to C$ is a computably enumerable set of natural numbers.
  This condition is always satisfied if $C$ has finite degree of imperfection $[C : C^p]$.
\end{remark}

\begin{remark}\label{rem:comparison-URt}
  In \cite[Section~6]{DittmannFehm_completions},
  the special case of a field $C$ separable algebraic over $C_0(t)$ for a perfect subfield $C_0 \subseteq C$ and $t$ transcendental over $C_0$ was considered.
  For valued fields $(F,v)$ containing $C$ as a trivially valued subfield,
  a special axiom $\mathrm{UR}(t)$ given by the $\mathcal{L}_{\mathrm{val}}(C)$-sentence
  \[ \forall s, r (r \not\in \mathcal{O} \to (t-s^p)r^2 \not\in \mathcal{O}) \]
  was considered there.
  It was observed in \cite[Remark~6.4]{DittmannFehm_completions} that this axiom is equivalent to demanding
  that $dt \otimes 1 \in \Omega_{\mathcal{O}_v} \otimes_{\mathcal{O}_v} Fv$ does not vanish.
  Given that $\{ t \}$ is a $p$-basis of $C$ and therefore $dt$ spans $\Omega_{C}$,
  Lemma~\ref{lem:fs-differential-criterion} implies that $\mathrm{UR}(t)$ exactly axiomatises that $\mathcal{O}_v$ is formally smooth over $C$.
  The equivalent condition from Lemma~\ref{lem:diff-indep-ad-hoc-crit}, which greatly simplifies for the field $C$ at hand,
  was also implicitly observed in \cite{DittmannFehm_completions}.
\end{remark}

The following proposition gives some connections between formal smoothness over $C$ and separability over $C$.
Recall that a (not necessarily algebraic) extension field $K$ of $C$ is separable over $C$ if and only if
every $p$-basis of $C$ remains $p$-independent in $K$;
equivalently, the natural map $\Omega_{C} \otimes K \to \Omega_{K}$ is injective \cite[Chapitre~V, §~16, No~4, Corollaire]{Bourbaki_algebre-4567}.
\begin{proposition}\label{prop:fs-separability}
  Let $\mathcal{O} \supseteq C$ be a valuation ring with maximal ideal $\mathfrak{m}$.
  \begin{enumerate}
  \item If $\mathcal{O}/\mathfrak{m}$ is separable over $C$, then $\mathcal{O}$ is formally smooth over $C$.
  \item If $\mathcal{O}$ does not have a uniformiser (i.e.\ $\mathfrak{m}$ is not principal) and $\mathcal{O}$ is formally smooth over $C$, then $\mathcal{O}/\mathfrak{m}$ is separable over $C$.
  \item If $\mathcal{O}$ is formally smooth over $C$, then the field of fractions $K = \Frac(\mathcal{O})$ is separable over $C$.
  \end{enumerate}
\end{proposition}
\begin{proof}
  For the first part, by the separability hypothesis the natural map
  $\Omega_{C} \otimes \mathcal{O}/\mathfrak{m} \to \Omega_{(\mathcal{O}/\mathfrak{m})}$
  is injective.
  However, this map factors through $\Omega_{\mathcal{O}} \otimes \mathcal{O}/\mathfrak{m}$,
  and so the criterion of Lemma~\ref{lem:fs-differential-criterion} implies formal smoothness.

  For the second part, $\mathcal{O}$ not having a uniformiser implies that the completion $\widehat{\mathcal{O}}$ is simply the field $\mathcal{O}/\mathfrak{m}$ with the discrete topology.
  Formal smoothness passes to the completion by Lemma~\ref{lem:fs-completion},
  and a field being formally smooth over $C$ means separability by the differential criterion from Lemma~\ref{lem:fs-differential-criterion}.

  For the third part, suppose first that $\mathcal{O}/\mathfrak{m}$ is separable over $C$.
  Then any $p$-basis of $C$ has $p$-independent residues in $\mathcal{O}/\mathfrak{m}$,
  and it follows easily that it must remain $p$-independent in $K$,
  so $K/C$ is separable.
  So now we may suppose that $\mathcal{O}/\mathfrak{m}$ is inseparable over $C$,
  and in particular, by the second part, that $\mathcal{O}$ has a uniformiser.
  The complete DVR $\widehat{\mathcal{O}}$ is formally smooth over $C$ by Lemma~\ref{lem:fs-completion},
  and therefore geometrically regular by \cite[Chapitre~0, Théorème~22.5.8]{EGA-IV-1}.
  In particular, its quotient field $\Frac(\widehat{\mathcal{O}})$ is separable over $C$.
  For the finest proper coarsening $v^0$ of the valuation $v$, the residue field $Kv^0$
  embeds into $\Frac(\widehat{\mathcal{O}})$: indeed, we observed earlier that $\widehat{\mathcal{O}}$ is the completion
  of the valuation ring of the discrete valuation on $Kv^0$ induced by $v$.
  Therefore $Kv^0$ must be separable over $C$ since $\Frac(\widehat{\mathcal{O}})$ is,
  so $\mathcal{O}_{v^0}$ is formally smooth over $C$ by the first part,
  and $K$ must be separable over $C$ by the first case considered.
\end{proof}

\begin{remark}\label{rem:fs-separability}
The first part of Proposition~\ref{prop:fs-separability} means that our axiomatisation results below generalise
some earlier ones, in particular \cite[Corollary~4.16]{ADF_existential}.
On the other hand, the second part of the proposition means that our results are only new for valuations which
admit a uniformiser.
\end{remark}

In the same vein as Proposition~\ref{prop:fs-separability},
one can show that formal smoothness implies a certain weakening of separability
for the field extension $(\mathcal{O}/\mathfrak{m})/C$.
This weakening is described in the following lemma.
\begin{lemma}\label{lem:almost-separable}
  Let $L/K$ be an extension of fields of characteristic $p$.
  The following are equivalent:
  \begin{enumerate}
  \item The kernel of the natural map $\Omega_K \otimes_K L \to \Omega_L$ of $L$-vector spaces has dimension at most $1$.
  \item For every $p$-independent set $A \subseteq K$,
    any subset $A' \subseteq A$ which is $p$-independent in $L$ and maximal with respect to this property
    satisfies $\lvert A \setminus A' \rvert \leq 1$.
  \end{enumerate}
  For every $p$-basis $B$ of $K$, these conditions are further equivalent to the following:
  \begin{enumerate}[resume,label=\protect{(\arabic*$_B$)}]
  \item For every finite non-empty subset $B_0 \subseteq B$,
    there exists $c \in B_0$ such that $B_0 \setminus \{ c \}$ is $p$-independent in $L$.
  \end{enumerate}
  If $L/K$ is finitely generated of transcendence degree $d$, the conditions above are further equivalent to the following:
  \begin{enumerate}[resume]
  \item The $L$-vector space $\Omega_{L/K}$ has dimension at most $d+1$.
  \end{enumerate}
\end{lemma}
Note that condition (3$_B$) evidently passes from $L/K$ to subextensions $L'/K$,
and that it is an $\mathcal{L}_{\mathrm{ring}}(K)$-axiomatisable condition on $L$.
\begin{proof}
  Let $B$ be a $p$-basis of $K$.
  We use throughout that $p$-independence in a field is characterised by linear independence in the vector space of Kähler differentials
  \cite[Chapitre~V, §~13, No~2, Théorème~1]{Bourbaki_algebre-4567}.
  It is clear that (2) implies (3$_B$).
  Let us assume that (1) holds and prove (2).
  So let $A \subseteq K$ be $p$-independent in $K$.
  Then the elements $da \in \Omega_K$, $a \in A$, are $K$-linearly independent,
  and therefore the $da \otimes 1 \in \Omega_K \otimes L$ are $L$-linearly independent.
  For a subset $A' \subseteq A$, the condition of being $p$-independent in $L$
  is equivalent to the $da$, $a \in A'$, being $L$-linearly independent
  in $\Omega_L$,
  which in turn is equivalent to the $da \otimes 1$, $a \in A'$ being $L$-linearly independent
  modulo $\ker(\Omega_K \otimes L \to \Omega_L)$.
  Since the kernel in question is at most $1$-dimensional by assumption (1),
  the maximal subsets $A' \subseteq A$ with this property all satisfy $\lvert A \setminus A'\rvert \leq 1$
  as a matter of linear algebra.

  Let us now assume (3$_B$) and prove (1).
  For every finite non-empty subset $B_0 \subseteq B$,
  assumption (3$_B$) forces that the $db \in \Omega_L$, $b \in B_0$, span a subspace of $\Omega_L$ of dimension at least $\lvert B_0 \rvert - 1$.
  In other words, the kernel of $\Omega_K \otimes L \to \Omega_L$ intersected with the span of the elements $db \otimes 1 \in \Omega_K \otimes L$, $b \in B_0$,
  has dimension at most $1$.
  Given that the $db$, $b \in B$, span $\Omega_K$ and therefore the $db \otimes 1$ span $\Omega_K \otimes L$,
  it follows that the entire kernel of $\Omega_K \otimes L \to \Omega_L$ has dimension at most $1$.

  Let us now assume that $L/K$ is finitely generated of transcendence degree $d$.
  By \cite[Chapitre~V, §~16, No~6, Théorème~4]{Bourbaki_algebre-4567},
  the kernel of $\Omega_K \otimes_K L \to \Omega_L$ has dimension at most $1$ if and only if its cokernel has dimension at most $d+1$.
  This cokernel can be identified with $\Omega_{L/K}$ \cite[Proposition~16.2]{Eisenbud_comm-algebra}.
  This proves the equivalence of (4) and (1).
\end{proof}

\begin{remark}\label{rem:compare-kraft-inseparability}
  In the case of finitely generated $L/K$,
  the conditions of Lemma~\ref{lem:almost-separable}
  are equivalent to the inseparability of $L/K$ in the sense of \cite[p.~111, Definition~3)]{Kraft_inseparable-koerpererw} being at most $1$.
  This notion is only defined for finitely generated $L/K$.
  It would seem sensible to extend the definition by declaring the inseparability of $L/K$ to be the dimension of the kernel
  of $\Omega_K \otimes_K L \to \Omega_L$ for a general field extension.
  However, we shall not pursue this further here.
\end{remark}

\begin{lemma}\label{lem:fs-almost-separable}
  Let $\mathcal{O}$ be a valuation ring formally smooth over $C$, and let $\mathfrak{m}$ be its maximal ideal.
  Then the equivalent conditions of Lemma~\ref{lem:almost-separable}
  hold for the field extension $(\mathcal{O}/\mathfrak{m})/C$.
\end{lemma}
\begin{proof}
  We must show that the kernel of $\Omega_C \otimes \mathcal{O}/\mathfrak{m} \to \Omega_{(\mathcal{O}/\mathfrak{m})}$ is at most $1$-dimensional.
  This linear map factors through $\Omega_C \otimes \mathcal{O}/\mathfrak{m} \to \Omega_{\mathcal{O}} \otimes \mathcal{O}/\mathfrak{m}$,
  which is injective by Lemma~\ref{lem:fs-differential-criterion}.
  The remaining linear map $\Omega_{\mathcal{O}} \otimes \mathcal{O}/\mathfrak{m} \to \Omega_{(\mathcal{O}/\mathfrak{m})}$
  has kernel isomorphic to $\mathfrak{m}/\mathfrak{m}^2$ by Lemma~\ref{lem:conormal-sequence}.
  Since $\mathfrak{m}/\mathfrak{m}^2$ has dimension at most $1$ as an $\mathcal{O}/\mathfrak{m}$-vector space,
  this proves the claim.
\end{proof}

Let us give some examples of valuation rings formally smooth over $C$.
\begin{proposition}\label{prop:fs-dvrs-from-varieties}
  Let $V$ be a smooth $C$-scheme.
  For every codimension-$1$ point $x \in V$,
  the local ring $\mathcal{O}_{V,x}$ is a DVR which is formally smooth over $C$.
  In particular, for every irreducible polynomial $f \in C[X]$,
  the ring $C[X]_{(f)}$ (obtained from $C[X]$ by inverting every polynomial not divisible by $f$)
  is a DVR which is formally smooth over $C$.
\end{proposition}
\begin{proof}
  Smoothness at a point for a locally finite type morphism between locally noetherian schemes
  is equivalent to formal smoothness for the corresponding local rings with their local topologies
  \cite[Proposition~17.5.3]{EGA-IV-4}.
  Applying this to the structure morphism $V \to \Spec C$ at the point $x$
  (so that the relevant local rings are $\mathcal{O}_{V,x}$ with the $\mathfrak{m}_{V,x}$-adic topology
  and $C$ with the discrete topology)
  yields the first statement.
  The ``in particular'' is the first part applied to $V = \mathbb{A}^1$.
\end{proof}

Note that the examples from the proposition do not necessarily have residue field separable over $C$.
As a simple example, consider $C=\mathbb{F}_p(t)$, and consider the valuation ring $C[X]_{(f)}$ with $f=X^p-t$.
The associated valuation is the $(X^p-t)$-adic valuation on the field $C(X)$, with residue field $C(\unspacedrt[p]{t})$.
The completion of $C(X)$ with respect to this valuation is in fact isomorphic to the valued field $K_2$
from Example~\ref{ex:basic} (send $X$ to the element $\sqrt[p]{\pi + t} \in K_2$).
This shows that the valuation ring of $K_2$ is formally smooth over $C$,
while this is not the case for $K_1$ and $K_3$ from the same example,
as follows for instance from Remark~\ref{rem:comparison-URt}.

However, the examples obtained from Proposition~\ref{prop:fs-dvrs-from-varieties}
are not particularly interesting for the axiomatisation results we have in mind.
There is a much more interesting class of examples of formally smooth valuation rings for these purposes.
These are in some sense obtained as ``non-standard versions'' of the geometric examples above, for specific fields $C$.
We need some work to establish these.
\begin{lemma}\label{lem:dvr-lin-indep-differentials}
  Let $\mathcal{O}$ be a DVR with maximal ideal $\mathfrak{m}$.
  Let $t_1, \dotsc, t_n \in \mathcal{O}$ be distinct and $p$-independent in the quotient field of $\mathcal{O}$.
  Then the $dt_i \otimes 1 \in \Omega_{\mathcal{O}} \otimes \mathcal{O}/\mathfrak{m}$ are $\mathcal{O}/\mathfrak{m}$-linearly independent if and only if
  $\mathcal{O}[X_1, \dotsc, X_n]/(X_1^p - t_1, \dotsc, X_n^p - t_n) = \mathcal{O}[\unspacedrt[p]{t_1}, \dotsc, \unspacedrt[p]{t_n}]$ is a DVR.
\end{lemma}
\begin{proof}
  The condition on $p$-independence of the $t_i$ ensures the equality of rings given.
  Now see \cite[Chapitre~0, Théorème~22.5.4]{EGA-IV-1}.
\end{proof}

\begin{proposition}\label{prop:fs-from-ultraproduct}
  Let $R$ be a normal integral domain of characteristic $p$, finitely generated over a field,
  and let $C$ denote its fraction field.
  Let $\mathcal{U}$ be a non-principal ultrafilter on the set of codimension-$1$ points $\mathfrak{p} \in \Spec(R)$.
  For all such $\mathfrak{p}$, the localisation $R_{\mathfrak{p}}$ is a discrete valuation ring of $C$,
  and so we have a corresponding completion $\widehat{C}^{\mathfrak{p}}$ with valuation $\widehat{v}_{\mathfrak{p}}$.
  Then for the ultraproduct of completions
  \[ (E, w) = \prod_{\mathfrak p} (\widehat{C}^{\mathfrak{p}}, \widehat{v}_{\mathfrak{p}})/\mathcal{U}, \]
  the valuation ring $\mathcal{O}_w$ is formally smooth over $C$.
\end{proposition}
\begin{proof}
  Observe first that the statement makes sense in that $C$ is a trivially valued subfield of $E$.
  Indeed, for any given $x \in R$, $x \neq 0$,
  the codimension-$1$ primes $\mathfrak{p}$ with $x \not\in R_{\mathfrak{p}}^\times$
  are exactly the minimal primes above $x$, of which there are only finitely many.
  By writing an arbitrary element $x \in C^\times$ as a quotient of two elements of $R$,
  we deduce that $x \in R_{\mathfrak{p}}^\times$ for all but finitely many codimension-$1$ primes $\mathfrak{p}$.

  We show formal smoothness by checking the condition of Lemma~\ref{lem:fs-differential-criterion}.
  So it suffices to show that for some $p$-basis $T$ of $C$, the elements $dt \otimes 1 \in \Omega_{\mathcal{O}_w} \otimes Ew$ ($t \in T$)
  are linearly independent, since the $dt$ form a basis of $\Omega_C$.
  It suffices to check this property on finite tuples,
  so let $t_1, \dotsc, t_n \in C$ be some $p$-independent set.
  We wish to show that the elements $dt_i \otimes 1 \in \Omega_{\mathcal{O}_w} \otimes Ew$ are $Ew$-linearly independent.
  By localising $R$ at some finite set of elements,
  we can ensure that $R$ contains all the $t_i$ while only removing finitely many
  codimension-$1$ points, thus ensuring that the statement remains unchanged.
    
  Consider the ring $R' = R[\unspacedrt[p]{t_1}, \dotsc, \unspacedrt[p]{t_n}]$.
  Then $R'$ is a domain with fraction field $C[\unspacedrt[p]{t_1}, \dotsc, \unspacedrt[p]{t_n}]$,
  and $R'$ is finite over $R$.
  Let $\mathfrak{p}' \in \Spec(R')$ be a codimension-$1$ prime,
  and let $\mathfrak{p} = \mathfrak{p}' \cap R \in \Spec(R)$ be the prime below.
  Since the $p$-th power of any element in $R'$ lies in $R$,
  we see that $\mathfrak{p}'$ consists of precisely those $x \in R'$ with $x^p \in \mathfrak{p}$.
  Therefore in the ring $R_{\mathfrak{p}}[\unspacedrt[p]{t_1}, \dotsc, \unspacedrt[p]{t_n}]$ containing $R'$,
  every $x \in R' \setminus \mathfrak{p}'$ is invertible since $x^p \in R \setminus \mathfrak{p}$ is invertible in $R_{\mathfrak{p}}$.
  It follows that $R_{\mathfrak{p}}[\unspacedrt[p]{t_1}, \dotsc, \unspacedrt[p]{t_n}] = R'_{\mathfrak{p}'}$.

  Now suppose that $R'_{\mathfrak{p}'}$ is a DVR.
  Then the elements $dt_i \otimes 1 \in \Omega_{R_{\mathfrak{p}}} \otimes R/\mathfrak{p}$ are $R/\mathfrak{p}$-linearly independent by Lemma~\ref{lem:dvr-lin-indep-differentials}.
  Thus the only codimension-$1$ primes $\mathfrak{p} \in \Spec(R)$
  for which the elements $dt_i \otimes 1 \in \Omega_{R_{\mathfrak{p}}} \otimes R/\mathfrak{p}$ may fail to be $R/\mathfrak{p}$-linearly independent
  must lie below primes $\mathfrak{p}' \in \Spec(R')$ for which the localisation $R'_{\mathfrak{p}'}$ is not a DVR,
  i.e.\ non-regular points of the scheme $\Spec(R')$.
  However, the regular locus of $\Spec(R')$ is open since $R'$ is finitely generated over a field
  \cite[Corollaire~6.12.5]{EGA-IV-2}.
  Since the generic point $(0) \in \Spec(R')$ is trivially regular,
  this means that only finitely many codimension-$1$ points $\mathfrak{p}' \in \Spec(R')$ can fail to be regular.

  Thus for all but finitely many $\mathfrak{p}$,
  the elements $dt_i \otimes 1 \in \Omega_{R_{\mathfrak{p}}} \otimes R/\mathfrak{p}$ are $R/\mathfrak{p}$-linearly independent.
  The same holds if we consider the $dt_i \otimes 1$ as elements of $\Omega_{\widehat{R_{\mathfrak{p}}}} \otimes \widehat{R_{\mathfrak{p}}}/\mathfrak{p}\widehat{R_{\mathfrak{p}}}$,
  since these are naturally isomorphic algebras over the field $R/\mathfrak{p} = \widehat{R_{\mathfrak{p}}}/\mathfrak{p}\widehat{R_{\mathfrak{p}}}$ by Remark~\ref{rem:differentials-completion}.
  By Corollary~\ref{cor:fo-independent-differentials} and Łoś's theorem, we see that
  the $dt_i \otimes 1 \in \Omega_{\mathcal{O}_w} \otimes Ew$ are $Ew$-linearly independent, which we had to show.
\end{proof}

\begin{remark}
  We have not stated the proposition in the maximal possible generality
  to keep the statement relatively elementary.
  Note first that the hypothesis of $R$ being finitely generated over a field
  is only used in the proof to ensure that $R$ is Noetherian and that for the finite ring extension $R'$,
  the scheme $\Spec(R')$ has an open set of normal points (since normality is equivalent to regularity in codimension $1$).
  For this it suffices to assume that $R$ is a Nagata ring \cite[Corollaire~6.13.3]{EGA-IV-2}.
  
  The proposition holds more generally in the situation of
  an integral noetherian Nagata scheme $X$ with function field $C$ of characteristic $p$ and
  a non-principal ultrafilter on the set of codimension-$1$ points $x \in X$.
  There is an open affine normal subscheme $X' \cong \Spec(R)$ of $X$, which is itself Nagata.
  Since $X'$ is dense in $X$ and $X$ is noetherian,
  all but finitely many codimension-$1$ points of $X$ lie in $X'$.

  Therefore, even though not all $x$ are normal points,
  the ultrafilter is concentrated on the normal ones,
  and so the ultraproduct of the local rings $\mathcal{O}_{X,x}$ is an ultraproduct of DVRs
  and therefore itself a valuation ring,
  and the reduction to the affine case shows that it is formally smooth over $C$.
\end{remark}

We note the following corollary for function fields over an arbitrary base field of characteristic $p$.
For us, a function field $F/K$ without further qualification
is a finitely generated field extension of transcendence degree $1$.
Associated to it is the set of places $\mathbb{P}_{F/K}$, i.e.\ the set of non-trivial valuations on $F$
which are trivial on $K$, which we normalise to have value group precisely $\mathbb{Z}$.
This corollary will be the basis of our analysis in Section~\ref{sec:function-fields-pac}.
\begin{corollary}\label{cor:aa-completion-fs}
  Let $K$ be a field of characteristic $p$ and $C/K$ a finitely generated field extension
  of transcendence degree $1$.
  In other words, $C/K$ is a function field, not necessarily separable.
  Then for every non-principal ultraproduct $(E,w)$ over the completions $(\hat{C}^v, \hat{v})$
  of $C$ with respect to places $v \in \mathbb{P}_{C/K}$,
  the valuation ring $\mathcal{O}_w$ is formally smooth over $C$.
\end{corollary}
\begin{proof}
  Pick an arbitrary place $w$, and consider the holomorphy ring
  $R = \bigcap_{v \neq w} \mathcal{O}_v \subseteq C$.
  This is a Dedekind domain with fraction field $C$, finitely generated over $K$, whose non-zero prime ideals
  correspond bijectively to the places $v \in \mathbb{P}_{C/K} \setminus \{ w \}$,
  see \cite[Corollary~5.7.11 and the following paragraph]{VillaSalvador_algebraic-function-fields}.
  The statement now follows from the proposition.
\end{proof}

\begin{remark}
  The crux of the matter,
  and what makes the corollary somewhat delicate to prove (by means of the proposition),
  is that we assert formal smoothness \emph{over $C$}.
  If $C/K$ is separable, we also have formal smoothness over $K$,
  and this is much easier to prove --
  it follows since all but finitely many $\mathcal{O}_{\hat{v}_{\mathfrak{p}}}$ individually are
  (geometrically regular, and so) formally smooth over $K$.
\end{remark}

For a later application to the theory of all completions of a function field,
we formulate the following lemma related to Corollary~\ref{cor:aa-completion-fs} in the situation
of a separable function field.
\begin{lemma}\label{lem:ff-fs-exceptional-places}
  Let $K$ be a field of characteristic $p$ and $F/K$ a separable function field.
  Let $t \in F$ be a separating element, i.e.\ $F/K(t)$ is a finite separable field extension.
  Let $B_K$ be a $p$-basis of $K$, so $B_K \cup \{ t \}$ is a $p$-basis of $F$.

  Let $v$ be a $K$-trivial place of $F$.
  Assume that $v(t) \geq 0$ and $v$ is unramified in the field extension $F/K(t)$,
  i.e.\ $K(t)$ contains a uniformiser of $v$ and the residue field extension $Fv/K(t)v$ is separable.
  Then the elements $ds \otimes 1 \in \Omega_{\mathcal{O}_v} \otimes Fv$, $s \in B_K \cup \{ t \}$, are $Fv$-linearly independent,
  and there exists $c \in B_K \cup \{ t + \mathfrak{m}_v \} \subseteq Fv$ such that the elements $ds \in \Omega_{Fv}$, $s \in (B_K \cup \{ t + \mathfrak{m}_v \}) \setminus \{ c \}$ are $Fv$-linearly independent.

  The set of exceptional places $v$, i.e.\ those for which $v(t) < 0$ or $v$ is ramified in $F/K(t)$,
  is finite.
\end{lemma}
\begin{proof}
  Let $v$ be a place with $v(t) \geq 0$ which is unramified in the field extension $F/K(t)$.
  Write $v_0$ for the restriction of $v$ to $K(t)$.
  The valuation ring $\mathcal{O}_{v_0}$ is a localisation of $K[t]$,
  and by the unramifiedness assumption, $\mathcal{O}_v/\mathcal{O}_{v_0}$ is an étale extension of rings.
  Let $b_1, \dotsc, b_n \in B_K$.
  We wish to check that $dt \otimes 1, db_1 \otimes 1, \dotsc, db_n \otimes 1 \in \Omega_{\mathcal{O}_v} \otimes Fv$ are $Fv$-linearly independent.
  By Lemma~\ref{lem:dvr-lin-indep-differentials},
  we need to check whether $S=\mathcal{O}_v[X_1, \dotsc, X_n, Y]/(X_1^p - b_1, \dotsc, X_n^p - b_n, Y^p - t)$ is a DVR.
  This ring is the tensor product over $\mathcal{O}_{v_0}$ of $\mathcal{O}_v$ and $R=\mathcal{O}_{v_0}[t^{1/p}, b_1^{1/p}, \dotsc, b_n^{1/p}]$.
  Therefore $S/R$ is an étale extension of rings.
  Both $R$ and $S$ are still local rings, since an element is a non-unit if and only if its $p$-th power
  (an element of $\mathcal{O}_{v_0}$ or $\mathcal{O}_v$, respectively) is a non-unit, so the sum of two non-units remains a non-unit.
  The ring $R$ is a localisation of the principal ideal domain $K(b_1^{1/p}, \dotsc, b_n^{1/p})[t^{1/p}]$
  and therefore a DVR.
  Hence $S$ is also a DVR \cite[Tag~0AP2]{StacksProject}, as we had to show.

  The kernel of the natural map $\Omega_{\mathcal{O}_v} \otimes Fv \to \Omega_{Fv}$ sending $ds \otimes 1$ to $d(s + \mathfrak{m}_v)$ has dimension $1$ by Lemma~\ref{lem:conormal-sequence}.
  Since the elements $ds \otimes 1 \in \Omega_{\mathcal{O}_v} \otimes Fv$, $s \in B_K \cup \{ t \}$ are $Fv$-linearly independent by what we have shown,
  it follows that there is an element $c \in B_K \cup \{ t + \mathfrak{m}_v \}$ such that the $ds \in \Omega_{Fv}$, $s \in (B_K \cup \{ t + \mathfrak{m}_v \}) \setminus \{ c \}$
  are $Fv$-linearly independent.

  There are only finitely many places $v$ with $v(t) \neq 0$.
  Since $F/K(t)$ is a finite separable extension, the finiteness of the number of ramified places for $F/K(t)$ is a standard
  result in the theory of function fields;
  see for instance \cite[Corollary~5.6.4 and Proposition~5.6.9]{VillaSalvador_algebraic-function-fields}.
\end{proof}

We return to the situation of a general base field $C$ of characteristic $p$.
For our results on existential theories,
we will need the following embedding lemma for DVRs.
This is in parallel to the use of embedding lemmas for DVRs in the finitely ramified situation,
see for instance \cite[Lemma~4.6]{ADJ_finitely-ramified}.
\begin{lemma}\label{lem:dvr-embedding}
  Let $R$ and $S$ be DVRs containing the field $C$, with maximal ideals $\mathfrak{m}$ and $\mathfrak{n}$, respectively.
  Assume that $S$ is complete, and that both $R$ and $S$ are formally smooth over $C$.
  Let $\varphi \colon R/\mathfrak{m} \hookrightarrow S/\mathfrak{n}$ be an embedding over $C$.
  Then there exists an embedding $R \hookrightarrow S$ of local $C$-algebras inducing $\varphi$.
\end{lemma}
\begin{proof}
  We may as well assume that $R$ is also complete,
  by replacing it by its completion if necessary.
  We have a $C$-homomorphism $R \to R/\mathfrak{m} \hookrightarrow S/\mathfrak{n}$.
  It is continuous with respect to the $\mathfrak{m}$-adic topology on $R$ and the discrete topology on $S/\mathfrak{n}$.
  By \cite[Chapitre~0, Corollaire~19.3.11]{EGA-IV-1} (using formal smoothness of $R$),
  the $C$-homomorphism $R \to S/\mathfrak{n}$ factors
  as the composition of $S \to S/\mathfrak{n}$ with a continuous $C$-homomorphism $R \to S$.
  Continuity implies that this is a homomorphism of local rings.

  Suppose that the $C$-homomorphism $R \to S$ we have obtained is not injective.
  Then its kernel must be precisely $\mathfrak{m}$, the only non-zero prime ideal of $R$.
  Therefore we have found an embedding $R/\mathfrak{m} \hookrightarrow S$ over $C$.
  If $R/\mathfrak{m}$ is inseparable over $C$,
  it follows that the fraction field of $S$ is inseparable over $C$,
  which is incompatible with the formal smoothness of $S$ over $C$ by Proposition~\ref{prop:fs-separability}~(3).
  If $R/\mathfrak{m}$ is separable over $C$, we know that $R \cong (R/\mathfrak{m})[\![\pi]\!]$ over $C$ by
  the structure theory of complete DVRs (cf.~\cite[Theorem~7.8]{Eisenbud_comm-algebra}),
  so we can construct an embedding $R \hookrightarrow S$ by using the embedding $R/\mathfrak{m} \hookrightarrow S$ and sending $\pi$ to an arbitrary non-zero element of $\mathfrak{n}$.
\end{proof}

\begin{remark}
  Without the hypothesis that $S$ is formally smooth over $C$,
  the proof above only yields a local $C$-homomorphism $R \to S$, not necessarily injective.
  In fact, the hypothesis of $S$ being formally smooth over $C$ is necessary to obtain an embedding.
  As an example, consider $C = \mathbb{F}_p(s)$, $R = C[\![\pi]\!](\unspacedrt[p^\infty]{s+\pi})$ and $S = C[\![\varpi]\!](\unspacedrt[p^\infty]{s})$.
  Both $R$ and $S$ are complete DVRs with residue field isomorphic to $C(\unspacedrt[p^\infty]{s})$ over $C$.
  One verifies that $R$ is formally smooth over $C$,
  for instance by showing that $ds \otimes 1 \in \Omega_{R} \otimes R/\mathfrak{m}$ is non-zero using the criterion of Lemma~\ref{lem:diff-indep-ad-hoc-crit},
  and concluding by Lemma~\ref{lem:fs-differential-criterion}
  (using that $ds$ spans $\Omega_{C}$).

  Any local $C$-homomorphism $R \to S$ must send $\pi$ to an element $x \in \mathfrak{n}$ such that $s + x$ has all $p$-power roots;
  equivalently, since $s$ has all $p$-power roots in $S$, $x$ must have all $p$-power roots.
  However, the only such element $x \in \mathfrak{n}$ is $x=0$,
  and so the unique local $C$-homomorphism $R \to S$ is not injective.
\end{remark}

\begin{remark}\label{rem:dvr-embedding-ad-hoc-proof}
  Let us sketch how one can prove Lemma~\ref{lem:dvr-embedding} without using the formalism of formal smoothness,
  using only the condition of Lemma~\ref{lem:diff-indep-ad-hoc-crit}~(2) for $T$ a $p$-basis of $C$
  (which axiomatises formal smoothness by the combination of Lemma~\ref{lem:diff-indep-ad-hoc-crit} and Lemma~\ref{lem:fs-differential-criterion}).
  This is in the spirit of \cite[Proposition~6.2]{DittmannFehm_completions} which, as discussed in Remark~\ref{rem:comparison-URt},
  works in a similar context.

  So let us work in the situation of Lemma~\ref{lem:dvr-embedding} and assume for simplicity that $R/\mathfrak{m}$ and $S/\mathfrak{n}$ are inseparable over $C$.
  Take a $p$-basis $T$ of $C$.
  By the inseparability assumption, $T$ is not $p$-independent in $R/\mathfrak{m}$,
  and so we can find distinct elements $\gamma, t_1, \dotsc, t_n \in T$ such that $\gamma = \sum_\alpha a_\alpha^p t^\alpha$ for suitable $a_\alpha \in R/\mathfrak{m}$ indexed by multi-indices $\alpha$.
  By the condition Lemma~\ref{lem:diff-indep-ad-hoc-crit}~(2) applied to $T$,
  it must certainly be the case that $T \setminus \{ \gamma \}$ is $p$-independent in $R/\mathfrak{m}$.
  Similarly, by considering the images $b_\alpha \in S/\mathfrak{n}$ of the $a_\alpha$ under the given mapping $R/\mathfrak{m} \hookrightarrow S/\mathfrak{n}$,
  the same condition yields that $T \setminus \{ \gamma \}$ is $p$-independent in $S/\mathfrak{n}$.

  By the structure theory of complete DVRs,
  there exists a coefficient field $D_1$ of $R$, i.e.\ a subfield mapping isomorphically onto $R/\mathfrak{m}$ under the residue map,
  containing $T \setminus \{ \gamma \}$:
  see for instance \cite[Theorem~9]{Cohen_complete-local-rings} for the existence of a coefficient field,
  and it is clear from the proof, see \cite[Lemma~11 and p.~77]{Cohen_complete-local-rings},
  that this coefficient field can be chosen to contain a given $p$-independent set;
  cf.\ also \cite[Theorem~7.8]{Eisenbud_comm-algebra}.
  Similarly, there is a coefficient field $D_2$ of $S$ containing $T \setminus \{ \gamma \}$.
  We can view $D_2$ as an extension field of $D_1$ via $D_1 \cong R/\mathfrak{m} \hookrightarrow S/\mathfrak{n} \cong D_2$.

  For every $\alpha$, let $x_\alpha \in D_1 \subseteq R$ be the lift of $a_\alpha$ in $D_1$.
  By the condition Lemma~\ref{lem:diff-indep-ad-hoc-crit}~(2),
  the element $\hat{x} = \gamma - \sum_\alpha x_\alpha^p t^\alpha$ must be a uniformiser of $R$.
  We can similarly lift the $b_\alpha \in S/\mathfrak{n}$ to uniquely determined elements $y_\alpha \in D_2 \subseteq S$,
  yielding a uniformiser $\hat{y} = \gamma - \sum_\alpha y_\alpha^p t^\alpha \in S$.

  The ring $R$ is isomorphic to the power series ring $D_1[\![t]\!]$ via a map preserving $D_1$ and sending $t$ to the uniformiser $\hat{x}$.
  We can thus construct an embedding $\iota \colon R \to S$ by sending $\hat{x}$ to $\hat{y}$ and $D_1$ into $D_2$.
  It only remains to be shown that $\iota$ is the identity on $C$.
  Note that $\iota$ preserves $T \setminus \{ \gamma \}$ by construction.
  Since the $x_\alpha$ are sent to the $y_\alpha$, it follows that also $\gamma$ is preserved.
  It now follows from a standard argument
  (cf.\ \cite[last paragraph of the proof of Lemma~4.2]{ADJ_finitely-ramified} in the mixed characteristic situation)
  that $C$ must be preserved since the $p$-basis $T$ of $C$ is preserved.
  Indeed, given an $x \in C$, choose a large natural number $N$.
  We have $C = C^{p^N}[T]$ since $T$ is a $p$-basis, and so we may expand $x$ as a sum of products of monomials in $T$ with $p^N$-th powers.
  Since $\iota$ preserves $T$, let us consider an element $c^{p^N}$ for some $c \in C$.
  A priori $\iota$ may not preserve $c^{p^N}$, but we certainly have $\iota(c^{p^N}) = \iota(c)^{p^N}$,
  and $\iota(c) - c \in \mathfrak{n}$ since the residue map $R/\mathfrak{m} \hookrightarrow S/\mathfrak{n}$ induced by $\iota$ is exactly the one initially given and therefore preserves $C$.
  Therefore $\iota(c^{p^N}) - c^{p^N} = (\iota(c) - c)^{p^N} \in \mathfrak{n}^N$.
  Returning to the given element $x \in C$, we see that we must have $\iota(x) - x \in \mathfrak{n}^N$.
  Since $N$ was arbitrary, $\iota(x) = x$ follows.
\end{remark}

\begin{remark}\label{rem:compare-münsterlemma}
  Lemma~\ref{lem:dvr-embedding} is reminiscent of the mixed characteristic situation in \cite[Lemma~4.2]{ADJ_finitely-ramified},
  where the following was obtained (in a different formulation):
  For DVRs $(R, \mathfrak{m})$ and $(S, \mathfrak{n})$, both of characteristic zero and residue characteristic $p$,
  with $R$ unramified (i.e.\ $\mathfrak{m} = pR$) and $S$ complete,
  any embedding $R/\mathfrak{m} \hookrightarrow S/\mathfrak{n}$ lifts to an embedding $R \hookrightarrow S$.
  Additionally, given a set $T \subseteq R/\mathfrak{m}$ of $p$-independent elements,
  the embedding $R \hookrightarrow S$ can be chosen to send a given lift in $R$ of $T$
  to a given lift in $S$ of the image of $T$ in $S/\mathfrak{n}$.
  Here the condition on lifts of $T$ is to be seen as an analogue of the condition in Lemma~\ref{lem:dvr-embedding} that the lift has to preserve the subfield $C$.

  The proof in given in \cite[Lemma~4.2]{ADJ_finitely-ramified} (with an alternative sketched in \cite[Remark~4.3]{ADJ_finitely-ramified})
  was rather complicated.
  Using formal smoothness, a significantly simpler proof can be given as follows.
  Let $B \subseteq R$ be a lift of $T$ and $B' \subseteq S$ be a lift of the image of $T$ in $S/\mathfrak{n}$.
  Let $R_0 = \mathbb{Z}[B]_{(p)} \subseteq R$.
  This is a local subring of $R$.
  By construction, its residue field is $\mathbb{F}_p(T)$, of which $R/(p)$ is a separable extension field,
  i.e.\ $R/(p)$ is a formally smooth $R_0/(p)$-algebra \cite[Chapitre~0, Théorème~19.6.1]{EGA-IV-1}.
  Furthermore $R$ is flat over $R_0$ since all torsion-free modules over DVRs are flat.
  Therefore $R$ is formally smooth over $R_0$ by \cite[Chapitre~0, Théorème~19.7.1]{EGA-IV-1}.
  We may see $S$ as an $R_0$-algebra by sending the elements of $B$ to the elements of $B'$.
  We now have $R_0$-algebra morphisms $R \hookrightarrow R/(p) \to S/\mathfrak{n}$.
  By the lifting result \cite[Chapitre~0, Corollaire~19.3.11]{EGA-IV-1} for formally smooth morphisms,
  the morphism $R \to S/\mathfrak{n}$ must factor through an $R_0$-morphism $R \to S$.
  This morphism induces the given $R/(p) = R/\mathfrak{m} \hookrightarrow S/\mathfrak{n}$ by construction,
  and it sends $B$ to $B'$ since it is an $R_0$-morphism.
  This proves the statement.
\end{remark}

\section{Existential theories of formally smooth valuation rings}
\label{sec:main-thms}

In this section we deduce from the previous embedding statement for DVRs, Lemma~\ref{lem:dvr-embedding},
our main result on existential theories of arbitrary formally smooth valuation rings.
This is a standard approach (cf.\ for instance \cite[Section~4 and Section~5]{ADJ_finitely-ramified}),
but our strategy bears closest resemblance to the one used in \cite[Section~4]{ADF_existential}.
Let us first discuss the crucial axiom (R4).
\begin{enumerate}
\item[(R4)] \label{R4} Every field $K$
 which is existentially closed in $K(\!(t)\!)$
is existentially closed in every extension $F/K$ for which there exists a valuation $v$
on $F$ trivial on $K$ with residue field $Fv = K$.
\end{enumerate}
Fields $K$ which are existentially closed in $K(\!(t)\!)$ are known as \emph{ample} or \emph{large} fields.
This notion was first studied in \cite{Pop_embedding-problems-large}.
Axiom (R4) was introduced in \cite[Section~2]{ADF_existential},
where also an extensive discussion of its relation to other axioms can be found,
and in particular the fact that it is implied by Resolution of Singularities.
See also \cite[Section~1.4]{Kuhlmann_places-arb-char} for related older work.
Although we will make no use of them in the following,
it is perhaps useful to observe some further equivalences.
In the following, a variety over a field is an integral separated scheme of finite type.
\begin{proposition}\label{prop:R4-equiv}
  Axiom (R4) is equivalent to each of the following other axioms:
  \begin{enumerate}
  \item[(R3)] For every field $K$ and every finitely generated extension $F/K$ such that there exists
    a valuation $v$ on $F$ which is trivial on $K$ with residue field $Fv = K$,
    there also exists such a valuation which additionally has value group $vF = \mathbb{Z}$.
  \item[(NLU)] For every field $K$ and every finitely generated extension $F/K$ such that there exists
    a valuation $v$ on $F$ which is trivial on $K$ with residue field $Fv = K$,
    there exists a $K$-variety $X$ with function field $F$ which has a smooth $K$-rational point.
  \end{enumerate}
\end{proposition}
The name (NLU), for ``non-local uniformisation'', is somewhat tongue-in-cheek:
The point is that the only difference to (the birational version of) Local Uniformisation,
stated as (R2) in \cite[Section~2]{ADF_existential},
is that there the variety $X$ must be chosen in such a way that $v$ is centred on a regular point of $X$
(necessarily rational since $Fv = K$, and therefore smooth \cite[Proposition~17.15.1]{EGA-IV-4}).
In contrast, here there is a priori no connection between the rational points on $X$ and the given valuation $v$.
\begin{proof}
  The equivalence between (R4) and (R3) is part of \cite[Proposition~2.3]{ADF_existential}.
  For the equivalence between (R3) and (NLU),
  let $F/K$ be a finitely generated field extension with a valuation $v$ on $F$, trivial on $K$,
  with residue field $K$.
  The field extension $F/K$ is regular \cite[Lemma~2.6.9~(b)]{FriedJarden}.
  By \cite[Lemma~9]{Fehm_embeddings-ample},
  there exists a $K$-variety $X$ with function field $F$ with a smooth $K$-rational point
  if and only if there exists a valuation $v$ on $F$, trivial on $K$,
  with residue field $K$ and value group $\mathbb{Z}$.
\end{proof}

\begin{remark}\label{rem:lu-discrete}
  The proof of \cite[Lemma~9]{Fehm_embeddings-ample},
  which drives the equivalence between (R3) and (NLU) above,
  crucially relies on a restricted local uniformisation result due to Kuhlmann.
  In the version required here, this states that for every finitely generated field extension $F/K$,
  every valuation $v$ on $F$, trivial on $K$, with residue field $K$ and value group $\mathbb{Z}$,
  is centred on a smooth $K$-rational point of some $K$-variety with function field $F$.
  In \cite{Fehm_embeddings-ample}, this was cited from a preprint that was never published.
  However, the statement is also a special case of \cite[Theorem~1.5]{KuhlmannKnaf_every-place-admits-lu}.
\end{remark}

This finishes our preliminary discussion of (R4).
Recall that a \emph{section} of (the residue map of) a valued field $(K,v)$ is a one-sided inverse $\zeta \colon Kv \to \mathcal{O}_v$ to the residue map $\mathcal{O}_v \to Kv$.
A \emph{partial section} is a map $\zeta \colon k_0 \to \mathcal{O}_v$ defined on a subfield $k_0 \subseteq Kv$,
such that similarly the composition $k_0 \to \mathcal{O}_v \to Kv$ is the embedding map.
Note that (partial) sections can only exist when $(K,v)$ is of equal characteristic
(as they require that $\mathcal{O}_v$ contains a field),
and the data of a partial section is completely determined by its image,
a subfield of $\mathcal{O}_v$.
Following usual practice,
we normally write $\zeta \colon k_0 \to K$ for a partial section,
with the understanding that the image is contained in the valuation ring.

\begin{lemma}\label{lem:bivalued-exists-section}
  Let $(K,v^0,v)$ be a field with two henselian valuations $v^0$ and $v$,
  where $v$ refines $v^0$.
  (This can be seen as a structure in the language $\mathcal{L}_{\mathrm{ring}}$ together with two unary predicates for the two valuation rings.)
  For every partial section $\zeta_0 \colon k_0 \to K$ of the residue map of $v^0$
  with $Kv^0/k_0$ separable there exists an elementary extension
  $(K^\ast, (v^0)^\ast, v^\ast) \succ (K,v^0,v)$ with a section $\zeta \colon K^\ast (v^0)^\ast \to K^\ast$
  of the residue map of $(v^0)^\ast$ that extends $\zeta_0$.
\end{lemma}
\begin{proof}
  This is a variant of \cite[Proposition~4.5]{ADF_existential} with two valuations, as follows.
  By loc.~cit.\ applied to the henselian valued field $(K,v^0)$,
  we obtain an elementary extension $(L, w^0) \succ (K,v^0)$
  with a section $\zeta_1 \colon L w^0 \to L$ that extends $\zeta_0$.
  We now take a structure $(K^\ast,(v^0)^\ast,v^\ast,\zeta)$, where $\zeta \colon K^\ast(v^0)^\ast \to K^\ast$,
  such that $(K^\ast,(v^0)^\ast,\zeta) \succ (L,w^0,\zeta_1)$ and $(K^\ast,(v^0)^\ast,v^\ast) \succ (K,v^0,v)$.
  Finding such $(K^\ast,(v^0)^\ast,v^\ast,\zeta)$ is a simple application of Robinson's Joint Consistency Lemma \cite[Theorem~6.6.1]{Hodges_longer}.
  (Using heavier machinery in the form of the Keisler--Shelah Theorem,
  one can also obtain $(K^\ast,(v^0)^\ast,v^\ast,\zeta)$ by
  taking ultrapowers of $(L,w^0,\zeta_1)$ and $(K,v^0,v)$ which
  are isomorphic as field extensions of $K$ with a single valuation.)
  Now $(K^\ast,(v^0)^\ast,v^\ast)$ and $\zeta$ are as desired.
\end{proof}

\begin{lemma}\label{lem:e-theory-pass-to-dvr}
  Assume (R4).
  Let $(K,v)$ be equicharacteristic henselian valued with a uniformiser,
  let $v^0$ be the finest proper coarsening of $v$,
  and let $\overline v$ be the induced valuation with value group $\mathbb{Z}$ on the residue field $Kv^0$.
  Let $C \subseteq K$ be a subfield on which $v^0$ is trivial, and assume that $Kv^0/C$ is separable.
  Then $\Th_\exists^{\mathcal{L}_{\mathrm{val}}(C)}(K, v) = \Th_\exists^{\mathcal{L}_{\mathrm{val}}(C)}(Kv^0, \overline{v})$.
\end{lemma}
\begin{proof}
  This is a variant of \cite[Corollary 4.6]{ADF_existential}, which implies this for the existential $\mathcal{L}_{\mathrm{ring}}(C)$-theories.
  By Lemma~\ref{lem:bivalued-exists-section} applied to the trivial partial section
  $\zeta_0 \colon Kv^0 \supseteq C \hookrightarrow K$,
  there exists an elementary extension $(K^\ast, (v^0)^\ast, v^\ast) \succ (K, v^0, v)$
  with a section $\zeta \colon K^\ast (v^0)^\ast \to K^\ast$ that is the identity on $C$.
  Let $\pi \in K^\ast (v^0)^\ast$ be a uniformiser for the valuation $\overline{v}^\ast$ induced by $v^\ast$.
  The residue field $K^\ast (v^0)^\ast$ is henselian and therefore ample
  (see for instance \cite{Pop_henselian-implies-large}),
  so by (R4) the field $\zeta(K^\ast (v^0)^\ast)$ is existentially closed in $K^\ast$.
  Hence for every existential $\mathcal{L}_{\mathrm{ring}}(C)$-formula in one variable $\varphi(X)$ we have
  \[ K^\ast (v^0)^\ast \models \varphi(\pi) \iff \zeta(K^\ast (v^0)^\ast) \models \varphi(\zeta(\pi)) \iff K^\ast \models \varphi(\zeta(\pi)) . \]
  Observe that $\zeta(\pi) \in K^\ast$ is a uniformiser for $v^\ast$ by construction.

  Let $\psi$ be an existential $\mathcal{L}_{\mathrm{val}}(C)$-sentence.
  We may assume it is in negation normal form, i.e.\ negation symbols only occur immediately before atomic formulas.
  We construct from it an existential $\mathcal{L}_{\mathrm{ring}}(C)$-formula $\varphi(X)$ by replacing every
  positive occurrence of $x \in \mathcal{O}$ by $\exists y (Xx^2 = y^2 - y)$
  and every negative occurrence $x \not\in \mathcal{O}$ by $\exists z, y (zx = 1 \mathrel\land z^2 = X(y^2 - y))$.
  For every henselian valued field $(L,w)$ containing $C$
  with a uniformiser $\varpi$ we have $(L,w) \models \psi$ if and only if $L \models \varphi(\varpi)$,
  since the formulas (in the free variable $x$)
  \[ \exists y (\varpi x^2 = y^2-y) \quad \text{and} \quad \exists z,y (zx=1 \mathrel\land z^2 = \varpi (y^2-y)) \]
  define the valuation ring $\mathcal{O}_w$ and its complement, respectively.

  In particular, we have
  \[ (K^\ast (v^0)^\ast, \overline{v}^\ast) \models \psi \iff K^\ast (v^0)^\ast \models \varphi(\pi) \iff K^\ast \models \varphi(\zeta(\pi)) \iff (K^\ast,v^\ast) \models \psi ,\]
  finishing the proof.
\end{proof}

We can now state the first version of our main result.
For the remainder of the section, we work over a field $C$ of characteristic $p>0$.
(For $C$ of characteristic $0$, all results still hold, but do not require (R4) and follow
from the model theory of henselian valued fields of residue characteristic $0$;
compare \cite[Proposition~4.11 and Remark~4.18]{ADF_existential}.)
\begin{proposition}\label{prop:main-monotonicity}
  Assume (R4).
  Let $(K,v)$ and $(L,w)$ be non-trivially henselian valued and containing the trivially valued subfield $C$.
  Assume that the valuation rings $\mathcal{O}_v$ and $\mathcal{O}_w$ are formally smooth over $C$.
  If $\Th_\exists^{\mathcal{L}_{\mathrm{ring}}(C)}(Kv) \subseteq \Th_\exists^{\mathcal{L}_{\mathrm{ring}}(C)}(Lw)$,
  then $\Th_\exists^{\mathcal{L}_{\mathrm{val}}(C)}(K,v) \subseteq \Th_\exists^{\mathcal{L}_{\mathrm{val}}(C)}(L,w)$.
\end{proposition}
We will see later that the hypothesis that $\mathcal{O}_w$ be formally smooth over $C$ is unnecessary, see Remark~\ref{rem:monotonicity-half-fs};
on the other hand, the hypothesis that $\mathcal{O}_v$ be formally smooth over $C$ is essential, as is apparent from Example~\ref{ex:basic}.
\begin{proof}
  Passing to elementary extensions if necessary,
  we may assume without loss that both $(K,v)$ and $(L,w)$ are $\aleph_1$-saturated (although we will later destroy this assumption),
  and that there is an embedding $Kv \hookrightarrow Lw$ over $C$.
  Note that passing to elementary extensions preserves formal smoothness of the valuation rings (Corollary~\ref{cor:fo-fs}).

  If $Kv/C$ is separable, then by \cite[Corollary~4.16]{ADF_existential} (assuming (R4)),
  we may replace $(K,v)$ by $Kv(\!(t)\!)$ with the $t$-adic valuation,
  without changing its existential $\mathcal{L}_{\mathrm{val}}(C)$-theory.
  If $Kv/C$ is not separable,
  then Proposition~\ref{prop:fs-separability}~(2) implies that it has a uniformiser,
  and for the finest proper coarsening $v^0$ of $v$ and the induced valuation $\overline{v}$ with value group $\mathbb{Z}$ on the residue field $Kv^0$
  we have that $(Kv^0, \overline{v})$ is complete by the saturation assumption \cite[Lemma~7.14]{Dries_lectures-model-theory-valued-fields}.
  Thus on the level of valuation rings, the passage from $\mathcal{O}_v$ to $\mathcal{O}_{\overline{v}}$ is simply the passage to the $\mathfrak{m}_v$-adic completion.
  In particular, $\mathcal{O}_{\overline{v}}$ is still formally smooth over $C$ by Lemma~\ref{lem:fs-completion}.
  Therefore $Kv^0 = \Frac(\mathcal{O}_{\overline{v}})$ is separable over $C$ by Proposition~\ref{prop:fs-separability}~(3).
  By Lemma~\ref{lem:e-theory-pass-to-dvr} (again relying on (R4)), we may replace $(K,v)$ by the complete discretely valued field
  $(Kv^0, \overline{v})$, without changing the existential $\mathcal{L}_{\mathrm{val}}(C)$-theory.

  Thus we may assume that $(K,v)$ is complete discretely valued,
  and similarly that $(L,w)$ is complete discretely valued,
  without changing the residue fields $Kv$ and $Lw$.
  Now Lemma~\ref{lem:dvr-embedding} implies that we have an embedding $(K,v) \hookrightarrow (L,w)$ over $C$,
  yielding the statement.
\end{proof}

We note a number of consequences.
We write $\equiv_\exists$ for the notion of existential equivalence of structures, i.e.\ having the same existential theory.
We have the following Ax--Kochen/Ershov style characterisation of existential equivalence.
\begin{corollary}[Relative existential completeness]\label{cor:existential-ake}
  Assume (R4).
  Let $(K,v)$ and $(L,w)$ be non-trivially henselian valued, both containing the trivially valued subfield $C$.
  Assume that the valuation rings $\mathcal{O}_v$ and $\mathcal{O}_w$ are formally smooth over $C$.
  Then:
  \begin{align*}
    \underbrace{(K,v) \equiv_\exists (L,w)}_{\text{in }\mathcal{L}_{\mathrm{val}}(C)} \Longleftrightarrow \underbrace{Kv \equiv_\exists Lw}_{\text{in }\mathcal{L}_{\mathrm{ring}}(C)} 
  \end{align*}
\end{corollary}
\begin{proof}
  The fields $Kv$ and $Lw$ are existentially interpretable in $(K,v)$ and $(L,w)$, respectively, by the same formulas:
  Indeed, the residue field is given as the valuation ring modulo the equivalence relation defined by the formula
  $X=Y \mathrel\lor \exists z (z(X-Y) = 1 \mathrel\land z \not\in \mathcal{O})$ in variables $X,Y$.
  Therefore $(K,v) \equiv_\exists (L,w)$ in $\mathcal{L}_{\mathrm{val}}(C)$ implies that $Kv$ and $Lw$ have the same positive existential $\mathcal{L}_{\mathrm{ring}}(C)$-theory.
  For fields in a constant expansion of $\mathcal{L}_{\mathrm{ring}}$, positive existential theories determine the full existential theories,
  since inequalities can be eliminated (replacing $x \neq 0$ by $\exists y (xy=1)$).
  This yields the forward direction.
  The backward direction is immediate from Proposition~\ref{prop:main-monotonicity}.
\end{proof}

We also want to give consequences concerning decidability.
Recall that, for theories $T_1$ and $T_2$ in computable languages $\mathcal{L}_1$ and $\mathcal{L}_2$,
a many-one reduction from $T_1$ to $T_2$ is given by a computable mapping from $\mathcal{L}_1$-sentences to $\mathcal{L}_2$-sentences
such that an $\mathcal{L}_1$-sentence lies in $T_1$ if and only if the associated $\mathcal{L}_2$-sentence lies in $T_2$.
This is a formalisation of the notion that membership in $T_1$ is at most as hard to decide as membership in $T_2$.
In particular, if $T_1$ is many-one reducible to $T_2$ and $T_2$ is decidable, then also $T_1$ is decidable.
\begin{theorem}\label{thm:exist-elimination}
  Assume (R4).
  \begin{enumerate}
  \item (Existential elimination) For every existential $\mathcal{L}_{\mathrm{val}}(C)$-sentence $\varphi$
  there exists an existential $\mathcal{L}_{\mathrm{ring}}(C)$-sentence $\varphi^\ast$
  such that for every non-trivially henselian valued field $(K,v)$ containing $C$ as a trivially valued subfield,
  with $\mathcal{O}_v$ formally smooth over $C$,
  we have $(K,v) \models \varphi$ if and only if $Kv \models \varphi^\ast$.
  \end{enumerate}
  Assume furthermore that $C$ is countable and endowed with a bijection from a decidable subset of $\mathbb{N}$,
  yielding an injection $\mathcal{L}_{\mathrm{val}}(C) \to \mathbb{N}$ via a standard Gödel coding,
  such that the quantifier-free diagram of $C$ is decidable (i.e.\ $C$ is a computable field)
  and $C$ has a computably enumerable $p$-basis.
  \begin{enumerate}[resume]
  \item (Computable existential elimination)
    There exists a \emph{computable} mapping sending $\mathcal{L}_{\mathrm{val}}(C)$-sentences $\varphi$ to a suitable $\varphi^\ast$ as above.
  \item (Relative decidability) The existential $\mathcal{L}_{\mathrm{val}}(C)$-theory of any $(K,v)$ as above is many-one reducible to
    the existential $\mathcal{L}_{\mathrm{ring}}(C)$-theory of $Kv$.
    In particular, if the latter theory is decidable, then so is the former.
  \end{enumerate}
\end{theorem}
\begin{proof}
  We consider the $\mathcal{L}_{\mathrm{val}}(C)$-theory $T$ of valued fields $(K,v)$ as in the statement of the theorem
  (where we recall that formal smoothness is axiomatisable by Corollary~\ref{cor:fo-fs}).
  By Proposition~\ref{prop:main-monotonicity},
  for every $(K,v) \models T \cup \{ \varphi \}$ and every $(L,w) \models T \cup \{ \neg\varphi \}$,
  there exists an existential $\mathcal{L}_{\mathrm{ring}}(C)$-sentence $\psi$ with $Kv \models \psi$, $Lw \models \neg\psi$.
  For fixed $\varphi$, the existence of $\varphi^\ast$ (independent of $(K,v)$, $(L,w)$)
  is therefore an application of the standard Separation Lemma of model theory,
  see for instance \cite[Lemma~3.1.1]{TentZiegler}.

  Assume now we are in the situation of $C$ being a computable field with a computably enumerable $p$-basis.
  Then the theory $T$ above is computably enumerable
  (where the condition of a computably enumerable $p$-basis is needed to ensure a
  computably enumerable axiomatisation of formal smoothness, see Remark~\ref{rem:fo-fs-computable}).
  Given a formula $\varphi$ as above, we can then consider all candidate formulas $\varphi^\ast$
  and search for a proof in a sound and complete proof calculus that $\varphi$ holding in a model of $T$
  is equivalent to $\varphi^\ast$ holding in the residue field.
  This procedure is guaranteed to terminate.
  The third point is now immediate from the definition of many-one reducibility.
\end{proof}

\begin{remark}
  In the computable parts of the theorem, the condition that $C$ have a computably enumerable $p$-basis
  is automatic if $[C : C^p]$ is finite.
  We will later see that having a computably enumerable $p$-basis is part of the conditions of being a ``nicely behaved''
  computable field, see Remark~\ref{rem:nicely-computable}.
\end{remark}

\begin{remark}\label{rem:formal-bridges}
  As is apparent from the proof,
  Theorem~\ref{thm:exist-elimination} is largely a formal consequence of Proposition~\ref{prop:main-monotonicity}.
  See \cite[Proposition~2.18]{AnscombeFehm_syntactic-fragments} for a general formalism for this,
  although we prefer not to present our results in the abstract framework given there.
\end{remark}

\begin{remark}\label{rem:monotonicity-half-fs}
  Although we have no immediate use for this,
  let us (after the fact) give a proof that in Proposition~\ref{prop:main-monotonicity},
  the hypothesis that $\mathcal{O}_w$ is formally smooth over $C$ can be dispensed with.
  Indeed, assume (R4) and
  suppose that $(K,v)$ and $(L,w)$ are non-trivially henselian valued fields
  containing the trivially valued subfield $C$.
  Assume that $\mathcal{O}_v$ is formally smooth over $C$,
  and that $\Th_\exists^{\mathcal{L}_{\mathrm{ring}}(C)}(Kv) \subseteq \Th_\exists^{\mathcal{L}_{\mathrm{ring}}(C)}(Lw)$.
  We wish to show that $\Th_\exists^{\mathcal{L}_{\mathrm{val}}(C)}(K,v) \subseteq \Th_\exists^{\mathcal{L}_{\mathrm{val}}(C)}(L,w)$,
  so let $\varphi \in \Th_\exists^{\mathcal{L}_{\mathrm{val}}(C)}(K,v)$.
  Let $\varphi^\ast \in \Th_\exists^{\mathcal{L}_{\mathrm{ring}}(C)}(Kv)$ be a sentence corresponding to $\varphi$ under
  Theorem~\ref{thm:exist-elimination}~(1).
  After replacing $(L,w)$ by an elementary extension, we may assume that there is an embedding $Kv \hookrightarrow Lw$ over $C$.

  There is a finitely generated subextension $D/C$ of $Kv/C$ such that $\varphi^\ast$
  holds in $D$.
  Let $d$ be the transcendence degree of $D$ over $C$.
  The field $D$ can be generated by $d+1$ many elements over $C$:
  indeed, by Lemma~\ref{lem:fs-almost-separable} and Lemma~\ref{lem:almost-separable} we have $\dim_D \Omega_{D/C} \leq d+1$,
  so by \cite[Chapter~II, §~17, Theorem~41]{ZariskiSamuel_I},
  there exists $x_0, \dotsc, x_d \in D$ such that $D/C(x_0, \dotsc, x_d)$ is a finite separable extension.
  We may suppose after reordering that $x_1, \dotsc, x_d$ are algebraically independent, so $D/C(x_1, \dotsc, x_d)$ is finite.
  Its maximal separable subextension $D'/C(x_1, \dotsc, x_d)$ is also finite, and so by the Primitive Element Theorem
  in its strong version (see for instance \cite[Satz~12]{Bosch_algebra}), the extension $D'(x_0)/C(x_1, \dotsc, x_d)$
  is generated by a single element $x_0'$.
  The extension $D/D'(x_0)$ is both purely inseparable and separable by construction, so $D = D'(x_0) = C(x_0', x_1, \dotsc, x_d)$.

  Therefore $D$ can be realised as $\Frac(C[X_0, \dotsc, X_d]/(f))$ for some irreducible polynomial $f$.
  Applying Proposition~\ref{prop:fs-dvrs-from-varieties} to the codimension-$1$ point
  defined by $f$ on the scheme $\mathbb{A}^{d+1}_C$,
  we obtain a DVR $\mathcal{O} = C[X_0, \dotsc, X_d]_{(f)}$
  (i.e.\ the polynomial ring $C[X_0, \dotsc, X_d]$ with all elements not divisible by $f$ inverted)
  with quotient field $C(X_0, \dotsc, X_d)$ which is
  formally smooth over its subfield $C$.
  The embedding $C[X_0, \dotsc, X_d]/(f) \subseteq D \subseteq Kv \hookrightarrow Lw$
  lifts to a $C$-algebra homomorphism $h \colon C[X_0, \dotsc, X_d] \to \mathcal{O}_w$,
  determined by selecting suitable images for the $X_i$.
  By construction, the preimage of the maximal ideal $\mathfrak{m}_w$ of $\mathcal{O}_w$ under $h$ is precisely the ideal $(f)$.
  We may assume that $h(f) \neq 0$, possibly after modifying $h$ by adding an element of $\mathfrak{m}_w$ to some of the $h(X_i)$.
  Thus the kernel of $h$ is a prime ideal properly contained in $(f)$.
  The only such ideal is the zero ideal, i.e.\ $h$ is injective.
  By the universal property of localisations and henselisations, $h$ extends to $\mathcal{O}^h \supseteq \mathcal{O} \supseteq C[X_0, \dotsc, X_d]$,
  so we have a local $C$-embedding $\mathcal{O}^h \to \mathcal{O}_w$.
  For the valued field $(K_0,v_0)$ corresponding to $\mathcal{O}^h$, we therefore have an embedding of valued fields $(K_0,v_0) \hookrightarrow (L,w)$ over $C$.
  Since the residue field $K_0 v_0$ is $D$ by construction, we have $K_0 v_0 \models \varphi^\ast$,
  and therefore $\varphi \in \Th_\exists(K_0,v_0) \subseteq \Th_\exists(L,w)$.
\end{remark}

\section{Applications to function fields}
\label{sec:function-fields-pac}

In this section, we apply our previous results to analyse asymptotic theories of completions of function fields.
Problems of this kind were extensively studied in \cite{DittmannFehm_completions},
mostly for global function fields.
The setting here is that of a base field $K$ and a function field $F/K$.
Recall that for us, this just means that $F/K$ is finitely generated of transcendence degree $1$.
Recall also that we write $\mathbb{P}_{F/K}$ for the set of places of $F/K$.
We wish to know what sentences hold in all completions $\widehat{F}^v$, $v \in \mathbb{P}_{F/K}$, or in almost all completions, i.e.\ all but finitely many.
The situation for $K$ of characteristic zero is readily understandable with the usual Ax--Kochen--Ershov principles in residue characteristic zero,
so our attention is focused on $K$ of characteristic $p>0$.
We have to restrict our attention to \emph{universal/existential sentences},
i.e.\ boolean combination of existential sentences.
Note that while for a single field, the universal/existential theory is completely controlled by the existential theory,
the theory of all (or almost all) completions of $F/K$ is normally highly incomplete,
and the universal/existential theory encodes more information than its existential subpart.
The theory of universal/existential sentences holding in all (or almost all) completions
for instance encompasses information on whether a polynomial has a zero in all (almost all, respectively) completions,
and whether it has a zero in no (finitely many, respectively) completions.

A crucial piece of information in the setting above lies in the theories of the residue fields of $F/K$.
For each $v \in \mathbb{P}_{F/K}$, we view $Fv$ as an $\mathcal{L}_{\mathrm{ring}}(F)$-structure by interpreting the $\mathcal{L}_{\mathrm{ring}}$-symbols in the usual way,
and interpreting the constants for elements of $F$ by means of the residue map $F \to Fv$, which we take to be zero outside the valuation ring $\mathcal{O}_v$.%
\footnote{Every sentence involves only finitely many constants from $F$, each of which will lie in $\mathcal{O}_v$ for almost all $v \in \mathbb{P}_{F/K}$.
  Extending the residue map by zero outside the valuation ring $\mathcal{O}_v$ is therefore inessential if we are allowed to drop finitely many places,
  which is mostly the case for us.}
Therefore we can consider the $\mathcal{L}_{\mathrm{ring}}(F)$-\emph{theory of almost all residue fields of $F/K$},
i.e.\ the set of all $\mathcal{L}_{\mathrm{ring}}(F)$-sentences which hold in $Fv$ for all but finitely many $v \in \mathbb{P}_{F/K}$.
We now have:
\begin{proposition}\label{prop:theories-aa-completions-to-res-flds}
  Let $R_0$ be the $\mathcal{L}_{\mathrm{ring}}(F)$-theory of almost all residue fields of $F/K$.
  \begin{enumerate}
  \item Assume (R4).
    The universal/existential $\mathcal{L}_{\mathrm{val}}(F)$-theory of almost all completions $\widehat{F}^v$ is precisely the
    universal/existential theory of non-trivially henselian valued fields containing $F$ as a trivially valued subfield,
    with valuation ring formally smooth over $F$,
    with residue field satisfying the universal/existential part of $R_0$.

  \item Without assuming (R4):
    Suppose $R_0$ contains the theory of PAC fields.
    Then the universal/existential $\mathcal{L}_{\mathrm{ring}}(F)$-theory of almost all completions $\widehat{F}^v$ is the
    theory of non-trivially henselian valued fields containing $F$ as a trivially valued subfield,
    separable over $F$,
    with residue field satisfying the universal/existential part of $R_0$.
  \end{enumerate}
\end{proposition}
\begin{proof}
  This is mostly formal after our preceding work.
  Let $T_1$ be the $\mathcal{L}_{\mathrm{val}}(F)$-theory of valued fields $(E,w)$ containing $F$ as a trivially valued subfield,
  with $\mathcal{O}_w$ formally smooth over $F$.
  Every non-trivial ultraproduct of completions $\widehat{F}^v$ satisfies $T_1$ by Corollary~\ref{cor:aa-completion-fs},
  and it certainly satisfies $R_0$.
  Let $T_2$ be the class of $(E,w) \models T_1$ with residue field $Ew$ satisfying the universal/existential part of $R_0$.
  Now every sentence in $T_2$ holds in almost all $\widehat{F}^v$.
  Let conversely $(E,w)$ be an arbitrary model of $T_2$.
  By construction, we can find an ultrafilter $\mathcal{U}$ on $\mathbb{P}_{F/K}$ such that the corresponding ultraproduct over $Fv$
  has the same existential theory as $Ew$.
  Assuming (R4), the ultraproduct of $\widehat{F}^v$ along $\mathcal{U}$ has the same universal/existential theory as $(E,w)$ by Corollary~\ref{cor:existential-ake}.
  Therefore every universal/existential sentence holding in almost all completions $\widehat{F}^v$ holds in all models of $T_2$.
  This proves the first point.

  The second point is proved entirely analogously, substituting \cite[Proposition~5.1]{DittmannFehm_completions}
  for Corollary~\ref{cor:existential-ake}.
\end{proof}

\begin{corollary}\label{cor:aa-res-to-aa-completions}
  Let $K$ be a computable field with a computably enumerable $p$-basis.
  Let $F/K$ be a function field as above,
  and also consider it as a computable field.
  \begin{enumerate}
  \item Assume (R4).
    The universal/existential $\mathcal{L}_{\mathrm{val}}(F)$-theory of almost all completions $\widehat{F}^v$ is many-one reducible to
    the $\mathcal{L}_{\mathrm{ring}}(F)$-theory of almost all residue fields of $F/K$.
  \item Assume that $K$ is pseudo-algebraically closed, but do not assume (R4).
    Then we still have the same result for the universal/existential $\mathcal{L}_{\mathrm{ring}}(F)$-theory of almost all completions.
  \end{enumerate}
\end{corollary}
\begin{proof}
  The first point is immediate from the first point of Proposition~\ref{prop:theories-aa-completions-to-res-flds} and the second point of Theorem~\ref{thm:exist-elimination}.
  For the second point, one can formally copy the proof of Theorem~\ref{thm:exist-elimination} in the setting of
  $\mathcal{L}_{\mathrm{ring}}(F)$-theories of henselian non-trivially valued fields separable over the trivially valued subfield $F$
  (cf.\ also Remark~\ref{rem:formal-bridges}).
  Indeed, since two such valued fields have the same universal/existential $\mathcal{L}_{\mathrm{ring}}(F)$-theory as soon as their
  residue fields have the same universal/existential $\mathcal{L}_{\mathrm{ring}}(F)$-theory by \cite[Proposition~5.1]{DittmannFehm_completions},
  the Separation Lemma \cite[Lemma~3.1.1]{TentZiegler} yields for every universal/existential $\mathcal{L}_{\mathrm{ring}}(F)$-sentence $\varphi$
  another such sentence $\varphi^\ast$ such that $\varphi$ holds in such a valued field if and only if $\varphi^\ast$ holds in the residue field.
  A suitable $\varphi^\ast$ can be effectively found from $\varphi$ by enumerating proofs.
  It follows from the second part of Proposition~\ref{prop:theories-aa-completions-to-res-flds} that $\varphi$ holds
  in almost all completions if and only if $\varphi^\ast$ holds in almost all residue fields.
  This gives the desired many-one reduction.
\end{proof}

We now wish to obtain examples of function fields $F/K$ where the universal/existential theory of almost all completions is in fact decidable.
In \cite[Section~6]{DittmannFehm_completions}, this was done for $K$ finite, i.e.\ the case of global function fields.
It was crucial there that the universal/existential $\mathcal{L}_{\mathrm{ring}}(F)$-theory of almost all residue fields admits an explicit
description:
indeed, even the full $\mathcal{L}_{\mathrm{ring}}(F)$-theory of almost all residue fields
can be described as the (decidable) $\mathcal{L}_{\mathrm{ring}}(F)$-theory of pseudofinite fields extending $K$
(see \cite[Lemma~4.5]{DittmannFehm_completions}, summarising results of \cite[Chapter~20]{FriedJarden}).
To apply our results above, we similarly give explicit descriptions of the full $\mathcal{L}_{\mathrm{ring}}(F)$-theory of almost all residue fields
in the case of regular function fields $F/K$ over certain PAC fields $K$.
This is done in the following theorem, whose proof will occupy us for a significant part of the remainder of this section.

Recall here that the \emph{imperfect exponent} of a field $K$ of characteristic $p>0$ is the number $e \in \mathbb{N} \cup \{ \infty \}$
for which $[K : K^p] = p^e$.
Following \cite[Definition~18.5.7]{FriedJarden}, we call the field $K$ \emph{$f$-free} for $f \in \mathbb{N}$
if its absolute Galois group $G_K$ is a free profinite group on $f$ many generators.
Similarly, following \cite[Section~27.1]{FriedJarden},
we call it \emph{$\omega$-free} if every finite embedding problem for its absolute Galois group is solvable;
equivalently, if the absolute Galois group of some (or equivalently every) countable elementary subfield of $K$
is free on countably infinitely many generators \cite[p.~960]{Chatzidakis_properties-forking-omega-free-pac}.
If $K$ is $f$-free for some $0 \leq f \leq \omega$, then a finite extension field of $K$ of separable degree $n$ is $f'$-free
for $f' = 1 + n(f-1)$.
(Note that this makes sense in that necessarily $n=1$ if $f=0$, i.e.\ $K$ is separably closed,
and this formula is to be interpreted as $f' = \omega$ if $f=\omega$.)
Indeed, this is immediately translated into the profinite Nielsen--Schreier formula for
open subgroups of index $n$ of free profinite groups \cite[Proposition~17.6.2]{FriedJarden}.
\begin{theorem}\label{thm:almost-all-res-pac}
  Let $K$ be an $f$-free PAC field of characteristic $p>0$ for some $0 \leq f \leq \omega$,
  with imperfect exponent $0 \leq e \leq \infty$.
  Let $F/K$ be a regular function field in one variable.
  Then the models of the $\mathcal{L}_{\mathrm{ring}}(F)$-theory of almost all residue fields $Fv$
  are precisely the fields $L$ containing $F$ which satisfy:
  \begin{enumerate}
  \item $L$ is PAC.
  \item $L$ has imperfect exponent $e$, and furthermore,
    the equivalent conditions of Lemma~\ref{lem:almost-separable}
    are satisfied for the extension $L/F$.
  \item $L$ is $f'$-free for some $f'$, where:
    \begin{enumerate}
    \item $f' = 0$ if $f = 0$;
    \item $f' = 1$ if $f = 1$;
    \item $f' = \omega$ if $f = \omega$;
    \item if $1 < f < \omega$, then either $f' = \omega$,
      or $f' = 1 + n(f-1)$ with $n = [(K^{\mathrm{sep}} \cap L) \colon K]$.
    \end{enumerate}
  \end{enumerate}
\end{theorem}
The proof relies on the basic model theory of PAC fields.
Our results are interesting already for $f=0$, i.e.\ the case of separably closed base field $K$,
where (for $e>0$) our results of the previous section are key to deducing information
on the universal/existential theories of almost all completions
(since the method of \cite[Section~6]{DittmannFehm_completions} only works for perfect fields $K$).
However, Theorem~\ref{thm:almost-all-res-pac} also appears to be new for $e=0$, i.e.\ the case of perfect $K$;
in the case of finite $K$ (which is somewhat parallel to PAC fields),
the Chebotarev Density Theorem plays a crucial role, and a substitute for it is needed in the present case.

\begin{remark}\label{rem:aa-res-description-first-order}
  Let us observe straight away that conditions (1), (2) and (3) on extension fields $L/F$ from Theorem~\ref{thm:almost-all-res-pac} are all first-order axiomatisable
  in the language $\mathcal{L}_{\mathrm{ring}}(F)$;
  this will be used in the proof of the theorem.
  The property of being PAC is first-order \cite[Proposition~11.3.2]{FriedJarden}.
  Regarding (2), it is clear how to axiomatise the condition on the imperfect exponent,
  and for the conditions from Lemma~\ref{lem:almost-separable},
  we fix a $p$-basis $B$ of $F$ and $\mathcal{L}_{\mathrm{ring}}(F)$-axiomatise condition (3)$_B$ of that lemma.

  Note that condition (3) forces the absolute Galois group of $L$ to have the so-called embedding property,
  since either it is profinite free or $L$ is $\omega$-free,
  see \cite[Definition~24.1.2, Proposition~17.7.3]{FriedJarden}.
  There are first-order sentences axiomatising that the absolute Galois group of $L$ has the embedding property
  \cite[Paragraph~(5.14)]{Chatzidakis_properties-forking-omega-free-pac},
  and we use these as part of our axiomatisation of (3).
  Free profinite groups on $f<\omega$ many generators
  are uniquely determined by having precisely those finite groups as (continuous) quotients which can be generated
  by $f$ elements \cite[Proposition~16.10.7~(b)]{FriedJarden}.
  On the other hand, $L$ is $\omega$-free if and only if its absolute Galois group has the embedding property
  and admits all finite groups as quotients \cite[Theorem~24.8.1]{FriedJarden}.
  It is now clear how to axiomatise condition (3) from the theorem if $f \in \{ 0, 1, \omega \}$,
  since a finite group arises as a quotient of the absolute Galois group if and only if $L$
  has a finite Galois extension with this Galois group.
  If $1 < f < \omega$, we axiomatise condition (3) by expressing the following in a first-order way:
  Firstly, if the absolute Galois group of $L$ admits some finite group as quotient
  which cannot be generated by fewer than $f'$ elements for some $f'$,
  then every finite group which can be generated by $f'$ elements appears as a quotient.
  Secondly, if some separable irreducible polynomial over $K$ of degree $n$ has a zero in $L$,
  then all finite groups which can be generated by $f'=1+n(f-1)$ elements occur as
  quotients of the absolute Galois group of $L$.
  Thirdly, if the absolute Galois group of $L$ admits some finite group as quotient which
  cannot be generated by fewer than $f'$ elements for some $f'$,
  but it fails to admit some finite group as a quotient which can be generated by $f'+1$ elements,
  then $f' = 1 + n(f-1)$ for some $n$,
  and $L$ contains one of the finitely many separable extension of $K$ of degree $n$.
\end{remark}

The following lemma is somewhat reminiscent of the results on existentially definable subsets of ample fields of \cite{Anscombe_existentially-generated-subfields},
cf.\ also \cite[Theorem~1]{Fehm_subfields-ample}.
\begin{lemma}\label{lem:ample-primitive-element}
  Let $K_0$ be ample, $K/K_0$ a finite separable extension.
  Let $X/K$ be a smooth geometrically integral scheme of finite type, $f \colon X \to \mathbb{A}^1_K$ a non-constant morphism.
  If $X(K) \neq \emptyset$, then there exists $x \in X$ with $f(x) \in \mathbb{A}^1_K(K) = K$ a primitive element of $K/K_0$.
\end{lemma}
\begin{proof}
  We follow the Weil restriction strategy of \cite[Proposition~4.1]{JardenPoonen_galois-points-varieties}.
  Let $R$ denote the Weil restriction functor from schemes over $K$ to schemes over $K_0$.
  Applying it to $f$, we obtain a morphism $R(f) \colon R(X) \to R(\mathbb{A}^1_K) \cong \mathbb{A}^d_{K_0}$,
  where $d = [K : K_0]$.
  The scheme $R(X)$ is smooth since $X$ is smooth,
  and geometrically integral since this is the case for $X$.
  Furthermore $R(f)$ is flat:
  the morphism $f$ is flat as a non-constant morphism from an integral scheme to $\mathbb{A}^1_K$ \cite[Proposition~14.4]{GoertzWedhorn1},
  and étale Weil restriction preserves flatness
  (see \cite[Proposition~4.10.1 and Remark~4.10.2]{Scheiderer_real-and-etale-cohomology}, noting that
  flatness satisfies fpqc descent).
  In particular $R(f)$ is open, and thus has dense image.
  Note also that $R(X)$ has a (smooth) $K_0$-point, corresponding to any $K$-point on $X$.
  Given an intermediate field $K_1$, $K_0 \subseteq K_1 \subsetneq K$,
  the $K_1$-points on $\mathbb{A}^1_{K_1}$ correspond to $K_0$-points of $\mathbb{A}^d_{K_0}$
  in a linear proper subspace.
  Since there are only finitely many such intermediate fields, the union of the corresponding subspaces
  forms a closed proper subset of $\mathbb{A}^d_{K_0}$.
  Passing to preimages under $R(f)$,
  and noting that the set of $K_0$-points of $R(X)$ is Zariski dense by ampleness,
  we obtain a $K_0$-point $x_0 \in R(X)(K_0)$ whose image under $R(f)$ is not contained
  in the closed subset of $\mathbb{A}^d_{K_0}$ under discussion.
  Then $x_0$ corresponds to a point $x \in X(K)$ with $f(x)$ not in $\mathbb{A}^1_{K_0}(K_1)$
  for any $K_1$,
  i.e.\ $f(x)$ is a primitive element of $K/K_0$.
\end{proof}

We deduce the following sharpening of the PAC property (cf.\ \cite[Proposition~11.1.3]{FriedJarden}).
\begin{corollary}\label{cor:pac-primitive-element}
  Let $K_0$ be a PAC field, $K/K_0$ a finite separable extension.
  Let $R$ be an integral domain finitely generated over $K$ with quotient field regular over $K$.
  Let $s \in R \setminus K$.
  Then there exists a $K$-homomorphism $R \to K$ sending $s$ to a primitive element of $K/K_0$.
\end{corollary}
\begin{proof}
  The field $K$ is PAC \cite[Corollary~11.2.5]{FriedJarden}.
  Since the quotient field of $R$ is regular over $K$, $\operatorname{Spec} R$ is a geometrically integral $K$-variety.
  We may freely localise $R$ at finitely many elements.
  The variety $\operatorname{Spec} R$ is generically smooth,
  so after localising we may as well suppose that it is smooth.
  The element $s$ gives an embedding $K[x] \to R$, sending $x$ to $s$,
  and therefore a non-constant morphism $\operatorname{Spec} R \to \mathbb{A}^1_K$.
  By the PAC property, $\operatorname{Spec}(R)(K) = \operatorname{Hom}(R, K)$ is non-empty.
  The result now follows from Lemma~\ref{lem:ample-primitive-element}.
\end{proof}

The following lemma is a rather precise existence statement for homomorphisms $R \to K'$
onto finite extensions $K'/K$ for a PAC field $K$ and an integral domain $R$ finitely generated over $K$.
In other words, it gives closed points of integral geometrically reduced affine $K$-varieties $\Spec R$ with certain properties.
This is related to the role of the Chebotarev Density Theorem for varieties over finite fields.
It also seems related to Jarden's Čebotarev property of suitable PAC fields,
as discussed for instance in \cite[Remark after Corollary 14]{FriedHaranJarden_galois-stratification-frobenius},
but it does not seem to obviously follow from it.
\begin{lemma}\label{lem:separable-half}
  Let $K$ be an $f$-free PAC field of characteristic $p$ for some $0 \leq f \leq \omega$.
  Let $R$ be an integral domain finitely generated over $K$.
  Assume that the quotient field $L$ of $R$ is separable over $K$.
  Let an element $s \in R$ be given which is not algebraic over $K$.
  Let $p_1, \dotsc, p_n \in R[X]$ be monic polynomials
  which are separable and irreducible as elements of $L[X]$.
  Let $b_1, \dotsc, b_e \in R$ be elements which are $p$-independent in $L$,
  where $p^e \leq [K : K^p]$.
  Let $K' = L \cap K^{\mathrm{sep}}$, a finite separable extension of $K$.
  We write $f' = 1 + [K' : K] (f-1)$,
  so that $K'$ is $f'$-free.
  Assume that the Galois group of the splitting field of the $p_i$ over $L$
  can be generated by $f'$ elements.

  Then there exists a surjective $K$-homomorphism $h \colon R \to K'$
  such that $h(s)$ is a primitive element of $K'/K$,
  the images of the $p_i$ remain irreducible in $K'[X]$,
  and the images of the $b_i$ are $p$-independent in $K'$.
\end{lemma}
\begin{proof}
  For the case $K'=K$ (which is automatic if $f=0$), this follows from \cite[Lemma~24.1.1]{FriedJarden},
  as we now explain.
  Let $L'$ be the splitting field of the $p_i$ over $L$.
  We may freely localise $R$ at finitely many elements.
  Doing so, we may suppose that $R$ is integrally closed,
  that some specific primitive element of $L'/L$ is integral over $R$ and that the discriminant of its minimal polynomial
  is a unit of $R$.
  This means that the integral closure $S$ of $R$ in $L'$
  is a Galois cover of $R$ in the sense of \cite[Definition~6.1.3]{FriedJarden},
  associated to the Galois extension $L'/L$, see \cite[Lemma~6.1.2]{FriedJarden}.
  Since the $p_i$ are separable as polynomials over $K[X]$,
  the ideal in $K[X]$ generated by $p_i$ and its formal derivative $p_i'$ is the unit ideal for each $i$.
  Localising $R$ further, we may suppose that $p_i$ and $p_i'$ already generate the unit ideal in $R[X]$.\footnote{%
    One can show that this is in fact already implied by the Galois cover condition,
    since that forces $\Spec(R) \to \Spec(S)$ to be étale.}
  Since $L/K$ is separable and $K'=K$, the extension $L/K$ is regular.

  We now wish to apply \cite[Lemma~24.1.1]{FriedJarden}, with $F, E, E'$ in the notation there
  equal to $L', L, L$ in ours.
  Since $\Gal(K^{\mathrm{sep}}/K)$ is free on $f=f'$ many generators
  and $\Gal(L'/L)$ can be generated by $f'$ many elements,
  we can find an epimorphism $\tilde{\gamma} \colon \Gal(K^{\mathrm{sep}}/K) \to \Gal(L'/L)$
  whose composition with the restriction map $\Gal(L'/L) \to \Gal((L' \cap K^{\mathrm{sep}})/K)$ is
  the restriction $\Gal(K^{\mathrm{sep}}/K) \to \Gal((L' \cap K^{\mathrm{sep}})/K)$
  (this is an instance of \cite[Proposition~17.7.3]{FriedJarden} if $f$ is finite,
  and otherwise follows from the definition of $\omega$-freeness).
  Then, according to \cite[Lemma~24.1.1]{FriedJarden},
  there exists a surjective $K$-homomorphism $h \colon S \to M$ with $M$ the finite Galois extension of $K$
  associated to $\ker(\tilde{\gamma})$
  such that the images of the $b_i$ are $p$-independent in $M$,%
  \footnote{The lemma has ``$m$ bounded by $[K : K^p]$'', when it should be $p^m \leq [K : K^p]$, or equivalently ``$m$ bounded by the imperfect exponent of $K$''.}
  $h(R) = K$,
  and the associated decomposition group is $\Gal(L'/L)$.
  The last statement means that we have an isomorphism $\Gal(L'/L) \to \Gal(M/K)$
  given through the action of $\Gal(L'/L)$ on $S$ and passing to images under $h$,
  see \cite[Lemma~6.1.4]{FriedJarden}.
  Since $p_i$ and $p_i'$ generate the unit ideal in $R[X]$,
  similarly their images under $h$ generate the unit ideal in $K[X]$,
  i.e.\ the image of $p_i$ in $K[X]$ is a separable polynomial for each $i$.
  Since $\Gal(L'/L)$ acts transitively on the roots of each $p_i$,
  the same must happen in the image.
  Therefore the images of the $p_i$ in $K[X]$ are still irreducible.
  Now $h|_R \colon R \to K$ is the desired homomorphism.

  For the general case $K' \supseteq K$,
  we may assume that $K' \subseteq R$ by replacing $R$ by a localisation.
  We now want to work as in the first part of the proof, after replacing $K$ by $K'$.
  The only challenge is to ensure that the image of $s$ is a primitive element of $K'/K$
  in the application of \cite[Lemma~24.1.1]{FriedJarden}.
  For this, we inspect the proof of this lemma.
  It works by constructing an integral domain $U$, finite over $R$,
  whose quotient field $D$ is a finite separable extension of $L$, regular over $K'$.
  This $U$ is constructed (by a so-called field crossing argument)
  so that any $K'$-homomorphism $U \to K'$
  (which exists by the PAC property)
  induces a $K'$-homomorphism $S \to K'$ such
  the statement on decomposition groups above holds with respect to the ring extension $S/R$,
  and so the images of the $p_i$ under this homomorphism are irreducible in $K'[X]$.
  To ensure that the images of the $b_i$ remain $p$-independent under $U \to K'$,
  the proof of \cite[Lemma~24.1.1]{FriedJarden} invokes
  \cite[Proposition~10.11]{FriedJarden}.
  Inspection of this proof shows that this works by replacing $U$
  by another integral domain $U'$ (called $S$ in the proof),
  finite over $U$ and with quotient field regular over $K'$,
  such that any $K'$-homomorphism $U' \to K'$
  satisfies the desired statement on $p$-independence.
  By Corollary~\ref{cor:pac-primitive-element}, we can choose the homomorphism $U' \to K'$
  with the additional restriction that the image of $s$ is a primitive element of $K'/K$.
  Tracing back our steps, this ensures that all desired properties are satisfied.
\end{proof}

With a view towards generalising Lemma~\ref{lem:separable-half} beyond the case of $L/K$ separable,
we need the following easy lemma on $p$-independence in simple inseparable extensions of fields.
\begin{lemma}\label{lem:p-indep-simple-ext}
  Let $K$ be a field of characteristic $p$, let $x_1, \dotsc, x_n \in K$ be $p$-independent, and let $m \geq 1$.
  Then $x_1^{1/p^m}, x_2, \dotsc, x_n$ are $p$-independent in $K(x_1^{1/p^m})$.
\end{lemma}
\begin{proof}
  By \cite[Theorem~21]{MacLane_sep-trans-bases}, the elements $x_2, \dotsc, x_n$ are $p$-independent in $K(x_1^{1/p^\infty})$,
  and so a fortiori $p$-independent in its subfield $L := K(x_1^{1/p^m})$.
  Since $x_1^{1/p^m}$ is not contained in $L^p(x_2, \dotsc, x_n) \subseteq K(x_1^{1/p^{m-1}})$,
  it follows that $x_1^{1/p^m}, x_2, \dotsc, x_n$ must be $p$-independent in $L$.
\end{proof}

\begin{lemma}\label{lem:inseparablisation}
  Let $K$ be an $f$-free PAC field of characteristic $p$ for some $0 \leq f \leq \omega$.
  Let $R$ be an integral domain finitely generated over $K$ with quotient field $L$.
  Assume that the extension $L/K$ satisfies the equivalent conditions of Lemma~\ref{lem:almost-separable}.
  Let an element $s \in R \setminus L^p K$ be given which is not algebraic over $K$.
  Let $p_1, \dotsc, p_n \in R[X]$ be monic polynomials
  which are separable and irreducible as elements of $L[X]$.
  Let $s=b_1, \dotsc, b_e \in R$ be elements which are $p$-independent in $L$ over $K$,
  and let $c_1, \dotsc, c_{e'} \in K$ such that $b_1, \dotsc, b_{e}, c_1, \dotsc, c_{e'}$ are $p$-independent in $L$
  and $[K : K^p] \geq p^{e+e'}$.
  Let $K' = L \cap K^{\mathrm{sep}}$.
  We write $f' = 1 + [K' : K] (f-1)$,
  so that $K'$ is $f'$-free.
  Assume that the Galois group of the splitting field of the $p_i$ over $L$
  can be generated by $f'$ elements.

  Then there exists a purely inseparable finite extension $K''/K'$,
  a surjective $K$-homomorphism $h \colon R \to K''$
  such that $h(s)$ is a primitive element of $K''/K$,
  the images of the $p_i$ remain irreducible in $K''[X]$,
  and $h(b_1), \dotsc, h(b_{e}), h(c_1), \dotsc, h(c_{e'})$ are $p$-independent in $K''$.
\end{lemma}
\begin{proof}
  If $L/K$ is separable, we may take $K''=K'$ and reduce to Lemma~\ref{lem:separable-half}.
  Therefore assume that this is not the case.

  Replacing $R$ by a localisation,
  we may assume for each $i$ that the ideal in $R[X]$ generated by $p_i$ and its formal derivative $p_i'$
  is the unit ideal,
  in order to ensure that for each field with a homomorphism from $R$ the image of $p_i$ is separable.
  
  Let $d$ be the transcendence degree of $L/K$.
  Extend the list of $b_i$ to a relative $p$-basis $B$ of $L/K$.
  By our assumption of the conditions of Lemma~\ref{lem:almost-separable},
  the $L$-vector space $\Omega_{L/K}$ has dimension at most $d+1$.
  Since $L/K$ is inseparable, it must in fact have dimension exactly $d+1$ \cite[Chapitre~V, §~16, No~6, Corollaire~1]{Bourbaki_algebre-4567}.
  Therefore every relative $p$-basis of $L/K$, in particular $B$, has size $d+1$.

  Let $m>0$ be minimal such that $L^{p^m}K/K$ is separable:
  such $m$ exists since $L$ can be written as a finite purely inseparable extension of a finite separable extension
  of a purely transcendental extension of $K$,
  and so some $L^{p^m}$ is contained in a finite separable extension of a purely transcendental extension of $K$,
  and therefore itself separable over $K$.
  Since $B$ generates $L$ over $L^p K$, the set of $p^m$-th powers of elements of $B$
  generates $L^{p^m}K$ over $L^{p^{m+1}}K$.
  Since $L^{p^m}K/K$ has transcendence degree $d$,
  a relative $p$-basis has size $d$ \cite[Chapitre~V, §~16, No~6, Corollaire~1]{Bourbaki_algebre-4567}.
  The set of $p^m$-th powers of elements of $B$ therefore cannot be relatively $p$-independent in $L^{p^m}K$ over $K$,
  but it contains a relative $p$-basis.
  We may suppose that $\{ b^{p^m} \colon b \in B, b \neq s \}$ is a relative $p$-basis of $L^{p^m}K/K$:
  If it is not, then there must be an element $b^\ast \in B$, $b^\ast \neq s$,
  such that $\{ b^{p^m} \colon b \in B, b \neq b^\ast \}$ is a relative $p$-basis.
  Then either $(b^\ast)^{p^m}$ generates $L^{p^m}K$ over $L^{p^{m+1}}K(\{ b^{p^m} \colon b \in B, b \neq s, b \neq b^\ast \})$
  (in which case $\{ b^{p^m} \colon b \in B, b \neq s \}$ is the desired relative $p$-basis of $L^{p^m}K$ over $K$ after all),
  or $(b^\ast)^{p^m}$ lies in $L^{p^{m+1}}K(\{ b^{p^m} \colon b \in B, b \neq s, b \neq b^\ast \})$.
  In the latter case, $\{ (b^\ast + s)^{p^m} \} \cup \{ b^{p^m} \colon b \in B, b \neq s, b \neq b^\ast \}$
  is a relative $p$-basis, and we simply replace $b^\ast$ by $b^\ast + s$.
  This is permissible since it does not affect the conclusion:
  $p$-independence of a set of elements of any field is unaffected by replacing one element by its sum with another,
  much like for linear independence in a vector space.

  Let $L' = L^{p^m}K(\{ b \colon b \in B, b \neq s \})$.
  Since $L$ is generated over $L^{p^{m+1}}K$ by $B$,
  and $s^{p^m} \in L^{p^m}K = L^{p^{m+1}}K(\{b^{p^m}, b \in B, b \neq s \})$,
  we see that $L'$ is generated over $L'^pK$ by the $b \in B$, $b \neq s$.
  As the $b \in B$, $b \neq s$ are relatively $p$-independent over $K$ in $L$ and therefore in $L'$,
  it follows that the $b \in B$, $b \neq s$ must be a relative $p$-basis of $L'$ over $K$.
  Since the cardinality of this relative $p$-basis is equal to the transcendence degree $d$ of $L'/K$,
  $L'/K$ is separable \cite[Chapitre~V, §~16, No~6, Corollaire~1]{Bourbaki_algebre-4567}.

  Observe that $L = L'(s)$ is a simple purely inseparable extension of $L'$.
  We claim that $s^{p^m}, b_2, \dotsc, b_e, c_1, \dotsc, c_{e'}$ are $p$-independent in $L'$.
  First observe that $b_2, \dotsc, b_e, c_1, \dotsc, c_{e'}$ are certainly $p$-independent in $L'$
  since they are $p$-independent in $L$.
  Secondly, $s^{p^m}$ is not a $p$-th power in $L'$,
  since otherwise $L^{p^{m-1}}K \subseteq L'(s^{p^{m-1}})$ would be contained in $L'$ and therefore separable over $K$,
  in contradiction to the choice of $m$.
  If $s^{p^m}, b_2, \dotsc, b_e, c_1, \dotsc, c_{e'}$ were $p$-dependent in $L'$,
  we would therefore have $s^{p^m} \in L'^p(b_2, \dotsc, b_e, c_1, \dotsc, c_{e'}) \setminus L'^p$.
  In the field $L'(s^{p^{m-1}}) \subseteq L$, we therefore can write the $p$-th power $s^{p^m}$
  as a non-trivial polynomial in $b_2, \dotsc, b_e, c_1, \dotsc, c_{e'}$ with $p$-th power coefficients.
  This contradicts the $p$-independence of $b_2, \dotsc, b_e, c_1, \dotsc, c_{e'}$ in $L$.

  Now let $R' = R^{p^m}K[\{ b \in B, b \neq s \}]$.
  This is a finitely generated $K$-algebra with quotient field $L'$.
  We may extend it by finitely many generators in $L'$ to ensure that $R'[s] \supseteq R$.
  Write $p_i^{(p^m)}$ for the polynomial in $R'[X]$ obtained from $p_i$ by raising every coefficient to the $p^m$-th power.
  Note that $p_i^{(p^m)}$ and $p_i$ have the same splitting field over $L$
  (and indeed the same applies for their images over an arbitrary field with a homomorphism from $R$)
  since the roots of $p_i^{(p^m)}$ in an algebraic closure are the $p^m$-th powers of roots of $p_i$, and $p_i$ is separable.

  Now apply Lemma~\ref{lem:separable-half} to the $K$-algebra $R'$, the distinguished element $s^{p^m}$,
  the polynomials $p_i^{(p^m)}$,
  and the $p$-independent elements $s^{p^m}, b_2, \dotsc, b_e, c_1, \dotsc, c_{e'} \in L'$.
  We thus obtain a surjective $K$-homomorphism $R' \to K'$.
  Since $s^{p^m} \in R'$,
  this extends uniquely to a $K$-homomorphism $h \colon R'[s] \to K''$ for some extension field $K''$,
  which is finite purely inseparable and generated by the image of $s$ over $K'$.
  Since $K'$ is generated by $h(s^{p^m})$ over $K$,
  it is clear that $K''$ is generated by $h(s)$ over $K$ (not merely over $K'$).
  We now verify the desired properties for the $K$-homomorphism $h|_R \colon R \to K''$.
  By construction, the image of each $p_i^{(p^m)}$ is irreducible in $K'$,
  so the same holds for $p_i$ over $K'$
  (since an appropriate Galois group acts transitively on the set of roots of $p_i^{(p^m)}$
  and therefore must also act transitively on the set of roots of $p_i$),
  and so the image of $p_i$ over $K''$ is irreducible since the image of $p_i$ is separable
  and $K''/K'$ is purely inseparable.
  Finally, the images of $s^{p^m}, b_2, \dotsc, b_e, c_1, \dotsc, c_{e'}$ are $p$-independent in $K'$ by construction,
  and so the images of $s, b_2, \dotsc, b_e, c_1, \dotsc, c_{e'}$ are $p$-independent in $K'' = K'(h(s^{p^m})^{1/p^m})$
  by Lemma~\ref{lem:p-indep-simple-ext}.
\end{proof}

We need the following partial quantifier elimination result for PAC fields,
which is a routine derivation from the work of Cherlin, van den Dries, and Macintyre.
\begin{proposition}\label{prop:pac-normal-form}
  Let $F$ be a field of characteristic $p$, and let $0 \leq f \leq \omega$, $0 \leq e \leq \infty$.
  \begin{enumerate}
  \item
  Consider the $\mathcal{L}_{\mathrm{ring}}(F)$-theory of $f$-free PAC fields of imperfect exponent $e$ extending $F$.
  Modulo this theory, every $\mathcal{L}_{\mathrm{ring}}(F)$-sentence is equivalent to a disjunction
  of sentences of the form
  \begin{align*}
    \exists x_1, \dotsc, x_n \Big( & \bigwedge_{i = 1}^k f_i(x_1, \dotsc, x_n) = 0 \mathrel\land \\
    & \bigwedge_{i=1}^l (p_i(Y,x_1, \dotsc, x_n) \text{ has no root or is inseparable}) \mathrel\land \\
    & \bigwedge_{i = 1}^m \big((q_{i1}(x_1, \dotsc, x_n), \dotsc, q_{ij_i}(x_1, \dotsc, x_n)) \text{ are absolutely $p$-independent}\big) \Big)
  \end{align*}
  with $f_i \in F[X_1, \dotsc, X_n]$, $p_i \in F[Y, X_1, \dotsc, X_n]$, $q_{ij} \in F[X_1, \dotsc, X_n]$.
  (Here the properties described informally in words stand for $\mathcal{L}_{\mathrm{ring}}$-formulas in the obvious manner:
  e.g.\ a polynomial $p_i(Y,x_1, \dotsc, x_n) = a_dY^d + \dotsc a_1Y + a_0$ with $a_i \in F[X_1, \dotsc, X_n]$ having no root
  is an $\mathcal{L}_{\mathrm{ring}}$-expressible property of the $a_i$,
  and being separable means that $p_i$ is coprime to its formal derivative $\frac{\partial p_i}{\partial Y}$,
  which is likewise expressible in this way.)
  \item
  Suppose that $e$ is finite, fix $c_1, \dotsc, c_e \in F$,
  and consider the $\mathcal{L}_{\mathrm{ring}}(F)$-theory of $f$-free PAC fields of imperfect exponent $e$ extending $F$
  in which the $c_i$ form a $p$-basis.
  Modulo this theory, every $\mathcal{L}_{\mathrm{ring}}(F)$-sentence is equivalent to a disjunction
  of sentences of the form
  \begin{align*}
    \exists x_1, \dotsc, x_n \Big( & \bigwedge_{i = 1}^k f_i(x_1, \dotsc, x_n) = 0 \mathrel\land \\
    & \bigwedge_{i=1}^l (p_i(Y,x_1, \dotsc, x_n) \text{ has no root or is inseparable})
  \end{align*}
  with $f_i \in F[X_1, \dotsc, X_n]$ and $p_i \in F[Y, X_1, \dotsc, X_n]$.

  \item
  Suppose that $f$ is finite, fix an $f$-free subfield $K_0$ of $F$,
  and consider the $\mathcal{L}_{\mathrm{ring}}(F)$-theory of $f$-free PAC fields of imperfect exponent $e$ extending $F$
  in which $K_0$ is relatively separably closed.
  Modulo this theory, every $\mathcal{L}_{\mathrm{ring}}(F)$-sentence is equivalent to a disjunction
  of sentences of the form
  \begin{align*}
    \exists x_1, \dotsc, x_n \Big( & \bigwedge_{i = 1}^k f_i(x_1, \dotsc, x_n) = 0 \mathrel\land \\
    & \bigwedge_{i = 1}^m \big((q_{i1}(x_1, \dotsc, x_n), \dotsc, q_{ij_i}(x_1, \dotsc, x_n)) \text{ are absolutely $p$-independent}\big) \Big)
  \end{align*}
  with $f_i, q_{ij} \in F[X_1, \dotsc, X_n]$.

  \item
  In combination of the previous two points, suppose that both $e$ and $f$ are finite, fix $c_1, \dotsc, c_e \in F$
  and an $f$-free subfield $K_0$ of $F$,
  and consider the $\mathcal{L}_{\mathrm{ring}}(F)$-theory of $f$-free PAC fields of imperfect exponent $e$ extending $F$
  in which $K_0$ is relatively separably closed.
  Modulo this theory, every $\mathcal{L}_{\mathrm{ring}}(F)$-sentence is equivalent to a disjunction of sentences of the form
  \begin{align*}
    \exists x_1, \dotsc, x_n \Big( & \bigwedge_{i = 1}^k f_i(x_1, \dotsc, x_n) = 0 \Big)
  \end{align*}
  with $f_i \in F[X_1, \dotsc, X_n]$.
  (In particular, this theory is model complete.)
  \end{enumerate}
\end{proposition}
\begin{proof}
  Let us expand the language $\mathcal{L}_{\mathrm{ring}}(F)$ by an $(d+1)$-ary predicate $R_d$ for every $d \in \mathbb{N}$,
  and add axioms to the theory of fields expressing that $R_d(a_d, \dotsc, a_0)$ holds if and only if
  the polynomial $a_dY^d + \dotsb + a_1Y + a_0$ has no root or is inseparable.
  Furthermore, add a $d$-ary predicate $Q_d$ for every $d \geq 1$ and axioms expressing that $Q_d(a_1, \dotsc, a_d)$
  if and only if the $a_i$ are absolutely $p$-independent.
  In this way we obtain a definitional expansion of the $\mathcal{L}_{\mathrm{ring}}$-theory of fields
  and extension theories.
  Note that an extension of fields $K \subseteq L$ gives an extension of structures in the expanded language
  if and only if every separable polynomial in $K[X]$ without a root in $K$ does not gain a root in $L$
  (i.e.\ $K$ is relatively separably closed in $L$)
  and every $p$-independent tuple in $K$ remains $p$-independent in $L$ (i.e.\ $L$ is separable over $K$).

  By \cite[Theorem~5.15]{Chatzidakis_properties-forking-omega-free-pac} (after Cherlin--van den Dries--Macintyre),
  an extension $K \subseteq L$ of $f$-free PAC fields of imperfect exponent $e$ is elementary as soon as
  $K$ is relatively separably closed in $L$ and $L/K$ is separable.
  Therefore the theory of $f$-free PAC fields of imperfect exponent $e$ is model complete in our expanded language,
  and so every sentence is equivalent to an existential sentence in the expanded language.
  Note that negative occurrences of the predicates $R_d$ can be eliminated in favour of existential
  $\mathcal{L}_{\mathrm{ring}}$-formulas (since having a root is positively existentially expressible and
  being separable is quantifier-freely expressible, as the latter can be tested over the algebraic closure,
  where we have quantifier elimination).
  Likewise, negative occurrences of the predicate $Q_d$ can be eliminated since being $p$-dependent
  is existentially $\mathcal{L}_{\mathrm{ring}}$-expressible.
  Finally, polynomial inequalities can be rewritten as positive existential statements by replacing $x \neq 0$ by $\exists y (xy = 1)$.
  Therefore every sentence is equivalent to a positive existential sentence.
  The first statement follows.

  In the setting of finite exponent of imperfection $e$ with fixed $c_1, \dotsc, c_e$,
  any extension of fields $K \subseteq L$ which share the $p$-basis $c_1, \dotsc, c_e$ is automatically separable.
  Therefore the predicates $Q_d$ can be dispensed with.

  In the setting of finite $f$ with a $f$-free subfield $K_0$,
  for any extension of $f$-free fields $K \subseteq L$ in which $K_0$ is relatively separably closed,
  the subfield $K$ must necessarily be separably closed in $L$.
  Indeed, we have restriction morphisms of absolute Galois groups $G_L \to G_K \to G_{K_0}$.
  Each of these profinite groups is free on $f$ many generators by assumption,
  and both $G_L \to G_{K_0}$ and $G_K \to G_{K_0}$ are surjective since $K_0$ is relatively separably closed in $K$ and $L$.
  It follows that both these morphisms must in fact be isomorphisms \cite[Proposition~2.5.2]{RibesZalesskii_profinite-groups},
  and so $G_L \to G_K$ is also an isomorphism,
  wherefore $K$ is relatively separably closed in $L$.
  Thus in the setting above the predicates $R_d$ can be dispensed with.
\end{proof}

\begin{proof}[Proof of Theorem~\ref{thm:almost-all-res-pac}]
  We first show that every non-principal ultraproduct of residue fields $Fv$
  satisfies our conditions, and so all models of the $\mathcal{L}_{\mathrm{ring}}(F)$-theory of almost all $Fv$
  also satisfy them.
  Observe first that every finite extension $L/K$ is a PAC field of imperfect exponent $e$.
  For $n$ the separable degree of $L/K$, $L$ is $(1+n(f-1))$-free,
  and so $L$ satisfies condition (3).
  Therefore every non-principal ultraproduct of residue fields $Fv$ satisfies conditions (1) and (3)
  and has imperfect exponent $e$.
  Finally, every such ultraproduct is the residue field of a valuation ring formally smooth over $F$
  by Corollary~\ref{cor:aa-completion-fs},
  and therefore satisfies (2) by Lemma~\ref{lem:fs-almost-separable}.

  It remains to show that our conditions (1), (2), (3) entail the $\mathcal{L}_{\mathrm{ring}}(F)$-theory of almost all $Fv$.
  Equivalently, we take an arbitrary $\mathcal{L}_{\mathrm{ring}}(F)$-sentence $\varphi$
  that holds in some field $L/F$ satisfying the conditions,
  and need to show that there is a model of the theory of all residue fields of $F/K$
  additionally satisfying $\varphi$.
  Let $0 \leq f' \leq \omega$ be such that $L$ is $f'$-free.
  By Proposition~\ref{prop:pac-normal-form},
  $\varphi$ is equivalent modulo the theory of $f'$-free PAC fields extending $F$ to a disjunction of sentences $\psi$ of the form
  \begin{align*}
    \exists x_1, \dotsc, x_n \Big( & \bigwedge_{i = 1}^k f_i(x_1, \dotsc, x_n) = 0 \mathrel\land \\
    & \bigwedge_{i=1}^l (p_i(Y,x_1, \dotsc, x_n) \text{ has no root or is inseparable}) \mathrel\land \\
    & \bigwedge_{i = 1}^m \big((q_{i1}(x_1, \dotsc, x_n), \dotsc, q_{ij_i}(x_1, \dotsc, x_n)) \text{ are absolutely $p$-independent}\big) \Big)
  \end{align*}
  with $f_i \in F[X_1, \dotsc, X_n]$, $p_i \in F[Y, X_1, \dotsc, X_n]$, $q_{ij} \in F[X_1, \dotsc, X_n]$.
  We therefore now have $L \models \psi$.

  In the remainder of the proof,
  we proceed to construct certain $K$-algebras related to $L$.
  Most prominently, we will build a finitely generated subextension $E/F$ of $L/F$
  which in some sense approximates $L$ and in particular satisfies $E \models \psi$.
  We will explain the construction of these algebras below, but already indicate
  that they fit in the following diagram:
  \[\xymatrix{
      & & & & L \\
      & & R \ar@{^{(}->}[rr]^{\text{quot. field}} \ar@{->>}[ddll]_h & & E \ar@{^{(}->}[u] \\
      & & R_0 \ar@{^{(}->}[u]_{\text{f.g. algebra}} \ar@{^{(}->}[rr]^{\text{quot. field}} \ar@{->>}[dl] & & F \ar@{^{(}->}[u]_{\text{f.g. field ext.}} \\
      K'' \ar@{=}[r]& Fv & & K \ar@{^{(}->}[ul]|-{\text{f.g. algebra}} \ar@{^{(}->}[ur]_{\text{f.g. field ext.}} \ar@{^{(}->}[ll]^{\text{fin. field ext.}}&   
  }\]

  Since $F/K$ is separable, there exists a separating element $t \in F$ over $K$,
  i.e.\ an element such that $F/K(t)$ is a finite separable extension.
  Let $R_0 \ni t$ be an arbitrary integrally closed finitely generated $K$-subalgebra of $F$ with quotient field $F$
  (in particular $R_0$ is a holomorphy ring of $F/K$
  in the sense of \cite[Section~3.3]{FriedJarden} by \cite[Proposition~3.3.4]{FriedJarden}),
  and let $R$ be an $R_0$-subalgebra of $L$ generated by the coefficients of the $f_i, p_i, q_{ij}$
  and some tuple of witnesses $x_1, \dotsc, x_n$ of $\psi$ in $L$.
  Let $E$ be the fraction field of $R$.
  This is a finitely generated field extension of $K$ containing $F$.
  Note that we have $E \models \psi$, using the same witnesses $x_1, \dotsc, x_n$ as in $L$.

  We will in the following further enlarge $R$ (and correspondingly $E$) by adding finitely many additional generators.
  Write $\hat{p_i} = p_i(X, x_1, \dotsc, x_n) \in R[X]$ for each $i$, and
  $\hat{q_{ij}} = q_{ij}(x_1, \dotsc, x_n) \in R$ for all $i,j$.
  If some $\hat{p_i}$ is inseparable, i.e.\ shares a non-constant factor with its formal derivative in $L[X]$ and therefore in $E[X]$,
  we enlarge $R$ inside $E$ to ensure that the same holds in $R[X]$.
  Enumerate the monic irreducible factors in $L[X]$ of those $\hat{p_i}$ which are separable as $r_j \in L[X]$.
  Since these $\hat{p_i}$ have no root in $L$, the $r_j$ are not linear.
  Enlarging $R$ and $E$ inside $L$, we may assume $r_j \in R[X]$.
  Since $L$ is $f'$-free, the Galois group of the joint splitting field of the $r_j$
  over $L$ can be generated by $f'$ many elements.
  We may enlarge $R$ and $E$ inside $L$ to ensure that the corresponding statement also holds over $E$.

  If $K^{\mathrm{sep}} \cap L$ is a finite extension of $K$ (which will be the case if $1 < f' < \omega$ by the axioms),
  we additionally enlarge $R$ so that $K^{\mathrm{sep}} \cap L \subseteq R$.
  Let $K' = R \cap K^{\mathrm{sep}}$, which is a separable finite extension of $K$.
  Note that $K'$ is $f'$-free unless $1 < f < \omega$ and $f'$ is infinite.
  If we should be in this particular case,
  we can still ensure that $K'$ is $f''$-free for an arbitrarily large natural number $f''$:
  to do so, we first ensure that $[(K^{\mathrm{sep}} \cap L) : K]$ is large
  by replacing $L$ by a finite separable extension given
  by adjoining a generator of a suitable finite separable extension of $K$.
  This does not disturb conditions (1), (2) and (3), since (1) and (2) are preserved
  under passing to finite separable extensions, and any finite extension of $\omega$-free $L$
  remains $\omega$-free and therefore satisfies (3).
  Secondly, we can ensure that the enlarged $L$ still satisfies the formula $\psi$
  with the same witnesses $x_1, \dotsc, x_n$, simply by ensuring that none
  of the polynomials $r_j$ become reducible, which is clearly possible
  by a wise choice of extension.
  Through this enlargement of $L$, we have ensured that $[(K^{\mathrm{sep}} \cap L) : K]$
  is arbitrarily large,
  and we can now enlarge $R$ and $E$ to ensure that $K' = R \cap K^{\mathrm{sep}}$
  is of arbitrarily large degree over $K$ and so $f''$-free for arbitrarily large $f''$.
  This finishes our discussion of the case $1 < f < \omega$, $f'$ infinite,
  and we return to the general situation.

  Let $B_{L/K}$ be a relative $p$-basis of $L/K$;
  we may assume that $t \in B_{L/K}$ if $t \not\in L^pK$.
  Let $B_K$ be a $p$-basis of $K$, and let $B_K' \subseteq B_K$ be a maximal subset which is $p$-independent in $L$.
  Then $B_L := B_K' \cup B_{L/K}$ is a $p$-basis of $L$.
  If $L/K$ is separable, then $B_K' = B_K$, and otherwise $B_K \setminus B_K'$ contains precisely one element by hypothesis~(2).
  Note that we have $B_K \subseteq L^p(B_K')$ by maximality of $B_K'$
  and so $K = K^p(B_K) \subseteq L^p(B_K')$.
  If $t \in L^p K$, then it follows that the $p$-basis $B_K \cup \{ t \}$ of $F$
  is contained in $L^p(B_K')$,
  and so $B_K'$ is a maximal $p$-independent (in $L$) subset of $B_K \cup \{ t \}$.
  By hypothesis (2) (in terms of condition (2) of Lemma~\ref{lem:almost-separable}),
  this implies that $B_K'$ cannot be a proper subset of $B_K$
  and so (still under the extra condition that $t \in L^p K$) $L/K$ must be separable.
  We have $\hat{q_{ij}} \in L = L^p[B_L]$ for all $i,j$.
  Since for each $i$, the elements $\hat{q_{i1}}, \dotsc, \hat{q_{ij_i}} \in L$
  are $p$-independent,
  the exchange lemma for $p$-independence (see the discussion in \cite[p.~463]{MacLane_lattice-formulation})
  guarantees that there is a finite subset $B_i \subseteq B_L$
  of cardinality $j_i$ such that $L = L^p[B_L] = L^p[B_L \setminus B_i, \hat{q_{i1}}, \dotsc, \hat{q_{ij_i}}]$.
  We can find a finite subset $B^\ast \subseteq B_L$ containing all $B_i$,
  and enlarge $R$ and $E$ by finitely many generators inside $L$,
  to ensure that $B^\ast \subseteq R$
  and $R^p[B^\ast] = R^p[B^\ast \setminus B_i, \hat{q_{i1}}, \dotsc, \hat{q_{ij_i}}]$ for all $i$.
  If $t \not\in L^pK$ and so $t \in B_{L/K}$, we may enlarge $B^\ast$ to ensure that $t \in B^\ast$.

  We now construct a $K$-homomorphism $h$ from $R$ to a finite extension of $K'$.
  By our construction we have $t \in B^\ast$ unless $t \in L^pK$,
  in which case $L$ and its subfield $E$ are separable over $K$.
  Therefore Lemma~\ref{lem:separable-half} in the separable case and Lemma~\ref{lem:inseparablisation} otherwise
  produce a finite purely inseparable extension $K''/K'$
  and a surjective $K$-homomorphism $h \colon R \to K''$ such that $h(t)$ is a primitive element for $K''/K$,
  the images of the $r_j$ remain irreducible in $K''[X]$,
  and the images of $B^\ast$ are $p$-independent in $K''$.
  Restricting $h$ to $R_0$, we obtain a surjective $K$-homomorphism $R_0 \to K''$,
  corresponding (since $R_0$ is a holomorphy ring of $F/K$, see \cite[Proposition~3.3.2]{FriedJarden})
  to a place $v$ of $F/K$ with an identification $Fv \cong K''$.
  We claim that for this place $v$ we have $Fv \models \psi$, witnessed by $h(x_1), \dotsc, h(x_n)$.
  Observe first that the images of all the $\hat{p_i}$ in $K''$ have no root or are inseparable:
  for every $i$, either $\hat{p_i}$ is inseparable and therefore shares a factor with its formal derivative in $R[X]$,
  so the same holds for the images in $K''$,
  or $\hat{p_i}$ factors into some non-linear irreducible factors among the $r_j$,
  whose images remain irreducible in $K''$, so the image of $\hat{p_i}$ has no linear factor.
  Now we need to check that for each $i$, the tuple $(h(\hat{q_{i1}}), \dotsc, h(\hat{q_{ij_i}}))$
  is $p$-independent in $K''$.
  By construction, the images of the elements $B^\ast$ in $K''$ are contained in
  $h(R^p[B^\ast \setminus B_i, \hat{q_{i1}}, \dotsc, \hat{q_{ij_i}}]) \subseteq K''^p[h(B^\ast \setminus B_i), h(\hat{q_{i1}}), \dotsc, h(\hat{q_{ij_i}})]$.
  If $(h(\hat{q_{i1}}), \dotsc, h(\hat{q_{ij_i}}))$ were $p$-dependent,
  we see that the images of the $B^\ast$ in $K''$ would be contained in a subfield of $K''$
  generated over $K''^p$ by fewer than $\lvert B^\ast \rvert$ many elements,
  but this contradicts the $p$-independence of $h(B^\ast)$.
  Thus we have $Fv \models \psi$.
  
  Note that by varying the base ring $R_0$,
  we in fact obtain infinitely many places $v$ with $Fv \models \psi$.
  The separable part of $Fv/K$ is $K'$ by construction,
  and therefore $Fv$ is $f'$-free,
  or $f''$-free for arbitrarily large $f''$ in the case of $1 < f < \omega$, $f' = \omega$.
  By taking an ultraproduct of appropriate $Fv$,
  we therefore in all cases obtain a model of the theory of almost all residue fields of $F/K$
  which additionally satisfies $\psi$ and is $f'$-free.
  Therefore also $\varphi$ holds in this model, as desired.
\end{proof}

\begin{remark}\label{rem:pac-aa-completions-axiomatisable}
  In the situation of a function field $F/K$ as in Theorem~\ref{thm:almost-all-res-pac},
  we therefore obtain -- by combining the theorem with Proposition~\ref{prop:theories-aa-completions-to-res-flds} --
  an explicit axiomatisation of the universal/existential $\mathcal{L}_{\mathrm{ring}}(F)$-theory
  of almost all completions of $F$,
  and under (R4) also an explicit axiomatisation of the universal/existential $\mathcal{L}_{\mathrm{val}}(F)$-theory
  of almost all completions of $F$.
  However, note that the axioms involved here (for instance henselianity) are not all themselves universal/existential.

  Theorem~\ref{thm:almost-all-res-pac} requires that the function field $F/K$ is regular.
  If $F/K$ is only separable, we can however simply replace $K$ by its relative algebraic closure in $F$,
  which is itself an $f'$-free PAC field for suitable $f'$,
  and so we also obtain an axiomatisation of the universal/existential theory of almost all completions in this case.
\end{remark}

It is easy to see, assuming that the field $F$ is a suitably well-behaved computable field,
that the axioms for the $\mathcal{L}_{\mathrm{ring}}(F)$-theory described in Theorem~\ref{thm:almost-all-res-pac}
given in Remark~\ref{rem:aa-res-description-first-order} are computably enumerable,
and therefore the theory is computably enumerable.
However, with a view to applying Corollary~\ref{cor:aa-res-to-aa-completions},
we want to show that it is in fact a decidable theory.
For this we have to work harder.
We need the following fact after \cite[Section~23.1]{FriedJarden}, allowing us to construct PAC fields which extend a given field.
\begin{lemma}\label{lem:pac-construction}
  Let $0 \leq f \leq \omega$, and let $0 \leq e \leq \infty$.
  Let $K$ be a field of characteristic $p$.
  \begin{enumerate}
  \item Let finitely many separable polynomials $p_1, \dotsc, p_l \in K[Y]$ be given.
    Let $K'/K$ be the joint splitting field of the $p_i$.

    There exists an $f$-free PAC field $L$ of imperfect exponent $e$ extending $K$
    in which none of the $p_i$ have a root if and only if
    there exists an intermediate field $K \subseteq E \subseteq K'$ such that none of the $p_i$ have a root in $E$ and $\Gal(K'/E)$
    is generated by $f$ many elements.
    (If $f=\omega$, this condition just means that none of the $p_i$ have a root in $K$.)

    Furthermore, under this condition:
    \begin{enumerate}
    \item If $e=\infty$, then $L/K$ can be chosen to be separable.
    \item If $e<\infty$ and $p$-independent elements $c_1, \dotsc, c_e \in K$ are given,
      then $L/K$ can be chosen to have $p$-basis $c_1, \dotsc, c_e$.
    \end{enumerate}
  \item Let $K_0 \subseteq K$ be a relatively separably closed subfield which is $f$-free.
    Then there exists an $f$-free PAC field $L$ of imperfect exponent $e$ extending $K$
    in which $K_0$ is relatively separably closed.
    Furthermore:
    \begin{enumerate}
    \item If $e=\infty$, then $L/K$ can be chosen to be separable.
    \item If $e<\infty$ and $p$-independent elements $c_1, \dotsc, c_e \in K$ are given,
      then $L/K$ can be chosen to have $p$-basis $c_1, \dotsc, c_e$.
    \end{enumerate}
  \end{enumerate}
\end{lemma}
\begin{proof}
  Let us first argue for (1).
  If a field $L$ as desired exists, then taking $E = L \cap K'$ shows that the given condition is necessary:
  indeed, $\Gal(K'/E)$ is a quotient of the absolute Galois group $G_L$ and therefore can be generated by $f$ many elements.

  Conversely, suppose that the given condition is satisfied with some intermediate field $E$.
  Then there exists an epimorphism of profinite groups $G \twoheadrightarrow \Gal(K'/E)$ with $G$ a free profinite group on $f$ many generators.
  By \cite[Theorem~23.1.1]{FriedJarden},
  there exists a separable extension field $L/K$ which is PAC, with an isomorphism between $G$ and the absolute Galois group of $L$ that
  turns $G \twoheadrightarrow \Gal(K'/E)$ into the restriction map $G_L \to \Gal(K'/E)$;
  in other words, $L$ is $f$-free and $L \cap K' = E$.
  In particular, none of the $p_i$ have a root in $L$.
  (In \cite[Theorem~23.1.1]{FriedJarden} as written, the field $L$ is made to be perfect, not separable over $K$;
  but it is clear from the proof that this is achieved after the fact by passing to a maximal purely inseparable extension,
  and one may as well construct $L$ to be separable over $K$.)

  This construction does not yet force $L$ to have imperfect exponent $e$.
  By replacing $K$ by a purely transcendental extension $K(t_1, t_2, \dotsc)$ as the very first step,
  we may always assume that $K$ and its separable extension $L$ have infinite imperfect exponent.
  This finishes the proof if $e=\infty$.
  Otherwise, we may suppose we are given $p$-independent elements $c_1, \dotsc, c_e \in K$ that we wish to turn into a $p$-basis.
  Then we first expand $c_1, \dotsc, c_e$ to a $p$-basis of $K$ by adding a further set of elements $D$.
  Then $L' = L(\{ \unspacedrt[p^\infty]{d} \colon d \in D \})$ is a purely inseparable extension of $L$ in which the $c_i$ form a $p$-basis
  \cite[Theorem~21]{MacLane_sep-trans-bases}.
  Furthermore, $L'$ remains PAC as an algebraic extension of $L$ and shares its absolute Galois group.
  This finishes the proof of (1).

  For (2), set $G = G_{K_0}$, and let $G \to G_K$ be a splitting of the restriction map $G_K \to G_{K_0}$
  (the latter is surjective since $K_0$ is relatively separably closed in $K$, and so a splitting exists
  since $G_{K_0}$ is a projective profinite group as $K_0$ is $f$-free).
  Replacing $K$ by a separable algebraic extension if necessary, we may assume that $G \to G_K$ is surjective.
  The remainder of the proof is as for (1).
\end{proof}

\begin{remark}\label{rem:nicely-computable}
  In the following, we work with ``nicely behaved'' computable fields $K$ of characteristic $p$.
  See \cite{StoltenbergHansenTucker_computable-rings-fields} for basics on computable algebra.
  The first condition beyond computability that needs to be satisfied for us is that $K$ must have a \emph{splitting algorithm}:
  one way of stating this is that there must be an algorithm which decides, given a one-variable polynomial over $K$,
  whether it has a zero in $K$.
  (See \cite[Section~3.2]{StoltenbergHansenTucker_computable-rings-fields};
  equivalently, condition (F) from \cite{Seidenberg_ConstrInAlgebra} must be satisfied.)
  This condition by itself is insufficient, as for example it does not guarantee that every finite extension of $K$
  also has a splitting algorithm.
  Therefore we also impose condition (P) from \cite{Seidenberg_ConstrInAlgebra}:
  there must be an algorithm which, given finitely many elements from $K$, decides whether they are (absolutely) $p$-independent.
  We also say that $K$ \emph{has decidable $p$-independence relation}.
  It is clear that this condition forces $K$ to have a computably enumerable $p$-basis
  (by simply ``greedily'' adding elements to a candidate $p$-independent set).
  One can in fact show, using the exchange property for $p$-dependence, that having a computably enumerable $p$-basis is also
  sufficient, although we will not use this fact.

  It now follows from \cite{Seidenberg_ConstrInAlgebra} that all usual algebraic operations can be done in a decidable fashion over $K$.
  Every finitely generated extension field $K(x_1, \dotsc, x_n)$ of $K$ is itself computable
  (even computably stable over $K$, i.e.\ there is only one choice of good indexing of $K(x_1, \dotsc, x_n)$ up to computable reïndexing)
  and satisfies the given conditions,
  and in fact all the algorithms are completely uniform, only requiring as input a set of generators of the ideal in $K[X_1, \dotsc, X_n]$
  of relations between the $x_i$.
  Furthermore, given an ideal in $K[X_1, \dotsc, X_n]$ in terms of finitely many generators, all usual questions are decidable:
  in particular, one can decide whether a given polynomial is in the ideal,
  one can decide whether the ideal is prime,
  and one can compute a primary decomposition and its associated primes.
  The latter in particular means that given polynomials $f_1, \dotsc, f_m \in K[X_1, \dotsc, X_n]$,
  one can compute a list of polynomial tuples $(f_{11}, \dotsc, f_{1i_1}), \dotsc, (f_{k1}, \dotsc, f_{ki_k})$ in $K[X_1, \dotsc, X_n]$,
  such that for each $1 \leq j \leq k$, the polynomials $f_{j1}, \dotsc, f_{ji_j}$ generate a prime ideal in $K[X_1, \dotsc, X_n]$
  which contains all the $f_i$,
  and conversely if $x_1, \dotsc, x_n$ in some extension field of $K$ are annihilated by the $f_i$,
  then they are annihilated by $f_{k1}, \dotsc, f_{ki_k}$ for some $k$.

  The notion of a computable field with a splitting algorithm and decidable $p$-independence relation
  seems to be the same as the notion of a presented field with elimination theory in \cite[Definition~19.2.8]{FriedJarden}.
  However, \cite[Chapter~19]{FriedJarden} is less precise concerning questions of uniformity among field extensions of $K$
  than will be required in the following.
\end{remark}

\begin{lemma}\label{lem:decide-almost-separable}
  Let $F$ be a computable field of characteristic $p$ with a splitting algorithm and decidable $p$-independence relation.
  Then there is an algorithm to decide the following:
  Given $n, m$ and polynomials $f_1, \dotsc, f_n \in F[X_1, \dotsc, X_m]$ generating a prime ideal,
  determine whether for the extension field $L=\Frac(F[X_1, \dotsc, X_m]/(f_1, \dotsc, f_n))$ of $F$
  the equivalent conditions of Lemma~\ref{lem:almost-separable} are satisfied.
\end{lemma}
\begin{proof}
  Let $R = F[X_1, \dotsc, X_m]/(f_1, \dotsc, f_n)$, so $L$ is the fraction field of $R$.
  We can compute the transcendence degree $d$ of $L/F$ \cite[paragraph~6]{Seidenberg_ConstrInAlgebra}.
  We have to decide whether $\Omega_{L/F}$ is at most $(d+1)$-dimensional.
  This is the generic fibre of $\Omega_{R/F}$, where $R = F[X_1, \dotsc, X_m]/(f_1, \dotsc, f_n)$.
  The $R$-module $\Omega_{R/F}$ is given explicitly in terms of finitely many generators and relations \cite[Section~16.1]{Eisenbud_comm-algebra}.
  Thus its generic fibre is given as an $L$-vector space by the same generators and relations,
  and computing its dimension is a matter of linear algebra over the computable field $L$.
\end{proof}

\begin{proposition}\label{prop:almost-all-res-fields-decidable}
  Let $K$ be a computable PAC field of characteristic $p$ with free absolute Galois group,
  a splitting algorithm, and decidable $p$-independence relation.
  Let $F/K$ be a regular function field.
  Then the $\mathcal{L}_{\mathrm{ring}}(F)$-theory of almost all residue fields of $F/K$ is decidable.
\end{proposition}
\begin{proof}
  Let $T$ be this theory.
  Let $e$ be the imperfect exponent of $K$, and let $0 \leq f \leq \omega$ such that $K$ is $f$-free.
  We will use the characterisation of $T$ from Theorem~\ref{thm:almost-all-res-pac}.
  We must check whether the set of consequences of $T$ is decidable.
  Alternatively (passing to negations) we must, given an $\mathcal{L}_{\mathrm{ring}}(F)$-sentence $\varphi$,
  decide whether $T \cup \{ \varphi \}$ is consistent.

  Let $T_0$ be the $\mathcal{L}_{\mathrm{ring}}(F)$-theory of PAC fields $L$ of imperfect exponent $e$ extending $F$
  such that the equivalent conditions of Lemma~\ref{lem:almost-separable} are satisfied for $L/F$;
  in other words, this is the theory axiomatised by axiom schemes (1) and (2) from Theorem~\ref{thm:almost-all-res-pac}.
  It is clear from Remark~\ref{rem:aa-res-description-first-order}
  that this theory has a computably enumerable axiomatisation
  (using that $F$ has a computably enumerable $p$-basis
  as the $p$-independence relation for finite tuples is decidable).

  In the following, we will occasionally make a case distinction between the cases of finite and infinite $e$.
  If $e$ is finite, we may fix a $p$-basis $c_0, \dotsc, c_e$ of $F$,
  and then condition (2) from Theorem~\ref{thm:almost-all-res-pac} is equivalent
  to the assertion that some $e$-element subset of $\{ c_0, \dotsc, c_e \}$ is a $p$-basis.
  By considering all reorderings of $c_0, \dotsc, c_e$ in parallel, we may thus -- in the case of finite $e$ --
  replace $T_0$ by the stronger theory of PAC fields of imperfect exponent $e$ extending $F$
  in which $c_1, \dotsc, c_e$ are a $p$-basis.

  Let a sentence $\varphi$ as above be given.
  We wish to test whether $T \cup \{ \varphi \}$ is consistent.
  We begin by checking whether there exists an $\omega$-free model of this theory.
  If $f=0$ or $f=1$, none such can exist, so assume for this step that $f>1$.
  It suffices to check whether there is an $\omega$-free model of $T_0 \cup \{ \varphi \}$
  (since any such automatically satisfies axiom scheme (3) from Theorem~\ref{thm:almost-all-res-pac}).
  We may assume that $\varphi$ is of the form given in Proposition~\ref{prop:pac-normal-form}.
  (We use here that $T_0$ is computably enumerable, so that a disjunction of formulas of this form
  equivalent to the original $\varphi$ can be effectively found.)
  So if $e<\infty$, we may assume that $\varphi$ is of the form
  \[ \exists x_1, \dotsc, x_n \Big( \bigwedge_{i = 1}^k f_i(x_1, \dotsc, x_n) = 0 \mathrel\land \bigwedge_{i=1}^l (p_i(Y,x_1, \dotsc, x_n) \text{ has no root or is inseparable}) \Big)\]
  for suitable $f_i, p_i$.
  We may assume that the $f_i$ generate a prime ideal:
  by Remark~\ref{rem:nicely-computable}, we can compute the prime ideals associated to a primary decomposition of the ideal $(f_1, \dotsc, f_k)$,
  in terms of generators, and then replace the ideal $(f_1, \dotsc, f_k)$ by each of the prime ideals in turn.
  (If the $f_i$ generate the unit ideal and so there are no associated primes, the algorithm terminates here, since no solution can be found.)
  Consider the field $L_0=\Frac(F[X_1, \dotsc, X_n]/(f_1, \dotsc, f_k))$.
  By Remark~\ref{rem:nicely-computable}, we can test if $c_1, \dotsc, c_e$ are $p$-independent in $L_0$.
  We can also test if each of the polynomials $p_i(Y, X_1, \dotsc, X_n) \in L_0[Y]$ has no root or is inseparable over $L_0$.
  If both of these questions have positive answers,
  then by Lemma~\ref{lem:pac-construction}~(1)
  there exists an $\omega$-free extension field $L/L_0$ in which none of the separable $p_i(Y, X_1, \dotsc, X_n) \in L_0[Y]$ have a root
  and $c_1, \dotsc, c_e$ are a $p$-basis.
  Such $L$ is a model of $T_0 \cup \{ \varphi \}$ and hence of $T \cup \{ \varphi \}$.
  Otherwise, the computably enumerable $\mathcal{L}_{\mathrm{ring}}(F \cup \{ X_1, \dotsc, X_n \})$-theory given by $T_0$,
  the assertion that a polynomial $g \in F[X_1, \dotsc, X_n]$ annihilates the $X_i$ if and only if $g$ is in the ideal generated by the $f_i$,
  and the assertion that each of the $p_i(Y, X_1, \dotsc, X_n)$ has no root or is inseparable,
  is inconsistent, given that each model of it is naturally an extension of $L_0$.
  We can then effectively find a contradiction, which in this case means that we can find a polynomial $g$ not in the ideal generated by the $f_i$
  such that for every witness $(x_1, \dotsc, x_n)$ of $\varphi$ as above, we must have $g(x_1, \dotsc, x_n) = 0$.
  So we may add $g$ to the $f_i$ and start over.
  This procedure will terminate since we cannot have a properly increasing infinite sequence of polynomial ideals.

  In the case $e=\infty$, we can work as above, but must use the different normal form for $\varphi$ from Proposition~\ref{prop:pac-normal-form}.
  Therefore we must check whether in the field $L_0$ the necessary tuples are $p$-independent,
  which we can since $L_0$ has decidable $p$-independence relation (Remark~\ref{rem:nicely-computable}).
  We must also check whether $L_0/F$ satisfies the condition of Lemma~\ref{lem:almost-separable},
  since this is part of $T_0$.
  This is decidable by Lemma~\ref{lem:decide-almost-separable}.
  This finishes our treatment of $\omega$-free models.

  If in the above we have found an $\omega$-free model of $T_0 \cup \{ \varphi \}$, we are done,
  so we may suppose there is no such model.
  This means that a certain computably enumerable extension theory of $T_0 \cup \{ \varphi \}$ is inconsistent.
  We can effectively find a contradiction,
  and therefore effectively find a certain finite $f_0$ such that $T_0 \cup \{ \varphi \}$ has no $f'$-free model for any $f' > f_0$.
  Therefore, to test the consistency of $T \cup \{ \varphi \}$, it suffices to test whether it has an $f'$-free model for finitely many $f'$.
  Fix one such $f'$.
  Let us first handle the cases $f' = 0$ and $f'=1$ (which can only occur for $f'=f$).
  Let us first assume that $e<\infty$, so we again may assume that $\varphi$ is of the form
  \[ \exists x_1, \dotsc, x_n \Big( \bigwedge_{i = 1}^k f_i(x_1, \dotsc, x_n) = 0 \mathrel\land \bigwedge_{i=1}^l (p_i(Y,x_1, \dotsc, x_n) \text{ has no root or is inseparable}) \Big)\]
  for suitable $f_i, p_i$.
  (Notice that for general $\varphi$, its normal form in this sense modulo the axioms of $f'$-freeness may differ from its normal form
  modulo the axioms of $\omega$-freeness.)
  We can now proceed almost exactly as before in the $\omega$-free case;
  however, to check the existence of a suitable $f'$-free PAC field extending $L_0$,
  we need to verify an additional Galois-theoretic condition from Lemma~\ref{lem:pac-construction}~(1) regarding the polynomials $p_i$.
  Nonetheless, it is clear that this is decidable.
  The case $e=\infty$ is entirely analogous.

  This leaves us to decide whether $T \cup \{ \varphi \}$ has an $f'$-free model for a given finite $f'>1$.
  Inspecting the axioms, this can only happen if $f$ is finite and $f' = 1 + n(f-1)$ for some $n$;
  furthermore, in that case the part of a model separably algebraic over $K$ necessarily has degree $n$.
  The field $K$ only has finitely many separable extensions of degree $n$,
  and we can effectively enumerate them (in terms of the minimal polynomial of a generator, say).
  So we may as well fix a separable extension $K'/K$ of degree $n$ and only consider models which contain $K'$ as a relatively separably closed subfield.
  Note that $K'$ itself is necessarily $f'$-free.
  So let us work in the $\mathcal{L}_{\mathrm{ring}}(FK')$-theory $T''$ of $f'$-free PAC fields extending the field compositum $FK'$
  in which $K'$ is relatively separably closed and the axioms of $T'$ are satisfied.
  It is clear that $T''$ is computably enumerable.
  We wish to check whether $T'' \cup \{ \varphi \}$ is consistent.
  If $e$ is finite, by Proposition~\ref{prop:pac-normal-form} we may suppose that $\varphi$ has the form
  \[ \exists x_1, \dotsc, x_n \Big( \bigwedge_{i = 1}^k f_i(x_1, \dotsc, x_n) = 0 \Big) \]
  for suitable $f_i \in FK'[X_1, \dotsc, X_n]$.
  Working as above, we may suppose the $f_i$ generate a prime ideal in $FK'[X_1, \dotsc, X_n]$,
  and consider the field $L_0 = \Frac(FK'[X_1, \dotsc, X_n]/(f_1, \dotsc, f_k))$.
  If the field $K'$ is separably closed in $L_0$ and $c_1, \dotsc, c_e$ are $p$-independent,
  then $L_0$ embeds into a model of $T''$ as desired by Lemma~\ref{lem:pac-construction}~(2).
  Note that we can check whether $K'$ is separably closed in its finitely generated extension field $L_0$
  by first determining the degree of its relative algebraic closure in $L_0$, and then testing whether finitely many
  separable polynomials over $K'$ have a zero in $L_0$.
  If $L_0$ does not embed into a model of $T''$ as desired,
  we proceed inductively as above, adding a new polynomial $g$ to the list of $f_i$.
  The case $e=\infty$ is the same, mutatis mutandis.
\end{proof}

\begin{corollary}\label{cor:pac-aa-completions-decidable}
  Let $K$ be a computable PAC field with free absolute Galois group,
  a splitting algorithm,
  and a computably enumerable $p$-basis.
  Let $F/K$ be a separable function field.
  \begin{enumerate}
  \item The universal/existential $\mathcal{L}_{\mathrm{ring}}(F)$-theory of almost all completions of $F$ is decidable.
  \item Assume (R4). Then also the universal/existential $\mathcal{L}_{\mathrm{val}}(F)$-theory of almost all completions of $F$ is decidable.
  \end{enumerate}
\end{corollary}
\begin{proof}
  The relative algebraic closure of $K$ in $F$ is a finite separable extension of $K$.
  As such, it is computable and has a splitting algorithm \cite[Theorem~3.2.4]{StoltenbergHansenTucker_computable-rings-fields},
  and a computably enumerable $p$-basis of $K$ remains a computably enumerable $p$-basis of this finite separable extension.
  It is also a PAC field with free absolute Galois group.
  Hence we may as well replace $K$ by this finite separable extension.
  Therefore assume that $K$ is relatively separably closed in $F$, i.e.\ $F/K$ is regular.
  Now the $\mathcal{L}_{\mathrm{ring}}(F)$-theory of almost all residue fields is decidable by Proposition~\ref{prop:almost-all-res-fields-decidable}.
  The result follows from Corollary~\ref{cor:aa-res-to-aa-completions}.
\end{proof}

\begin{remark}\label{rem:constructing-examples-pac}
  Let us indicate how to obtain examples of fields $K$ as in the corollary.
  On the one hand, these exist for abstract computability-theoretic reasons.
  Namely, for every prime $p$, $0 \leq e \leq \infty$ and $0 \leq f \leq \omega$,
  the theory of $f$-free PAC fields of characteristic $p$ and imperfect exponent $e$
  is decidable
  (see \cite[Theorem~30.6.1]{FriedJarden} for the perfect case;
  the general statement can be shown without difficulty along the lines of the
  proof of Proposition~\ref{prop:almost-all-res-fields-decidable}, but with substantial simplification),
  and therefore has decidable models $K$ (i.e.\ models with domain $\mathbb{N}$ with decidable elementary diagram)
  \cite[Theorem~4.1]{Harizanov_pure-computable-model-theory}.
  Such $K$ certainly have a splitting algorithm and decidable $p$-independence relation.

  On the other hand, one can also produce simple examples of suitable $K$ by hand.
  For instance, for every prime $p$ and $0 \leq e < \infty$, the separable closure $\mathbb{F}_p(t_1, \dotsc, t_e)^{\mathrm{sep}}$
  admits a natural indexing turning it into a computable field with splitting algorithm and the decidable $p$-independence relation
  (where splitting algorithm and decidability of $p$-independence are easily lifted from suitable finitely generated subfields).
  As a separably closed field, it is also PAC.
\end{remark}

We also obtain decidability results for the universal/existential $\mathcal{L}_{\mathrm{val}}(F)$-theory of all (as opposed to almost all)
completions of a function field.
\begin{proposition}\label{prop:all-completions-decidable}
  Assume (R4).
  Let $K$ be a computable PAC field with free absolute Galois group, a splitting algorithm, and decidable $p$-independence relation.
  Let $F/K$ be a separable function field.
  Then the universal/existential $\mathcal{L}_{\mathrm{val}}(F)$-theory of all completions of $F$ is decidable.
\end{proposition}
In the proof, it is necessary to deal with individual places of $F/K$, and in particular to allow places as inputs to algorithms.
For this one needs to code them in some way in terms of the computable field structure; the precise details here do not matter.
Since the valuation ring of a place is a localisation of a finitely generated $K$-algebra
at a (finitely generated) prime ideal, one way of presenting a place is by a suitable list of generators (elements of $F$).

We follow the strategy of \cite[Corollary~6.7]{DittmannFehm_completions}.
We need the following lemma describing the universal/existential theory of a single completion.
\begin{lemma}\label{lem:single-completion-decidable}
  Assume (R4).
  Let $F/K$ be a function field of characteristic $p$, let $v \in \mathbb{P}_{F/K}$ and let $\pi \in F$ be a uniformiser for $v$.
  The universal/existential $\mathcal{L}_{\mathrm{val}}(F)$-theory of the completion $\widehat{F}^v$
  is entailed by the axioms of henselian valued fields containing the valued field $(F,v)$,
  the axiom that $\pi$ is a uniformiser,
  and axioms expressing that $Fv$ is existentially closed in the residue field
  (i.e.\ that the residue field satisfies the universal/existential $\mathcal{L}_{\mathrm{ring}}(Fv)$-theory of $Fv$).

  If $K$ is a computable PAC field with free absolute Galois group, a splitting algorithm, and decidable $p$-independence relation,
  then this universal/existential $\mathcal{L}_{\mathrm{val}}(F)$-theory of the completion $\widehat{F}^v$
  is decidable.
  The decision procedure is uniform in $v$ in the sense that there is one algorithm which,
  given an $\mathcal{L}_{\mathrm{val}}(F)$-sentence and a place $v \in \mathbb{P}_{F/K}$,
  determines whether the sentence holds in $\widehat{F}^v$.
\end{lemma}
\begin{proof}
  The description of the universal/existential $\mathcal{L}_{\mathrm{val}}(F)$-theory of $\widehat{F}^v$ is immediate from \cite[Theorem~4.12]{ADF_existential}.
  For the decidability, it suffices to observe that the given axioms form a computably enumerable scheme,
  since then one can search for a derivation in a proof calculus of a given sentence or its negation.
  This computably enumerable nature of the axioms is clear once we observe that the universal/existential
  $\mathcal{L}_{\mathrm{ring}}(Fv)$-theory of $Fv$ is decidable.
  The field $Fv$ is a finite extension of $K$, and therefore itself a PAC field.
  It is $f$-free for some $f$ that can be determined from the absolute Galois group of $K$ and the separable degree of $Fv/K$.
  In fact, even the full $\mathcal{L}_{\mathrm{ring}}(Fv)$-theory of $Fv$ (i.e.\ its elementary diagram) is decidable.
  Indeed, a field $E/Fv$ is an elementary extension of $Fv$ if and only if it is PAC, $f$-free,
  separable over $Fv$, it has the same degree of imperfection as $Fv$,
  and $Fv$ is relatively algebraically closed in it \cite[Theorem~5.15]{Chatzidakis_properties-forking-omega-free-pac}.
  This is a computably enumerable set of axioms for the elementary diagram.
\end{proof}

\begin{proof}[Proof of Proposition~\ref{prop:all-completions-decidable}]
  Let $B_K$ be a computably enumerable $p$-basis of $K$, and let $t \in F$ be a separating element over $K$,
  so that $B_K \cup \{ t \}$ is a computably enumerable $p$-basis of $F$.
  Let $\varphi$ be a universal/existential $\mathcal{L}_{\mathrm{val}}(F)$-sentence.
  We wish to decide if it holds in all completions of $F$.
  By Proposition~\ref{prop:almost-all-res-fields-decidable} (using (R4)),
  we can decide if it holds in almost all completions of $F$.
  If no, then it certainly does not hold in all completions, and we are done.
  Otherwise, by Proposition~\ref{prop:theories-aa-completions-to-res-flds},
  the sentence $\varphi$ is implied by the following axioms for a valued field $(E,w)$ in the language $\mathcal{L}_{\mathrm{val}}(F)$:
  \begin{enumerate}
  \item $(E,w)$ is a henselian non-trivially valued field;
  \item $(E,w)$ contains $F$ as a trivially valued subfield;
  \item the valuation ring $\mathcal{O}_w$ is formally smooth over $F$,
    which can be axiomatised by asserting that the elements $ds \otimes 1 \in \Omega_{\mathcal{O}_w} \otimes Ew$, $s \in B_K \cup \{ t \}$, are $Ew$-linearly independent,
    cf.\ Corollary~\ref{cor:fo-independent-differentials} and Corollary~\ref{cor:fo-fs};
  \item the residue field $Ew$ is a model of the theory of almost all residue fields
    as axiomatised in Theorem~\ref{thm:almost-all-res-pac},
    cf.\ Remark~\ref{rem:aa-res-description-first-order}:
    this is $\mathcal{L}_{\mathrm{ring}}(F)$-axiomatised by the axioms for being PAC,
    axiom scheme (3) on absolute Galois groups from that theorem,
    the axiom of having the same imperfect exponent as $K$,
    and for every finite non-empty subset $B_0$ of the $p$-basis $B_K \cup \{ t \}$ of $F$
    the axiom that there exists an element $c \in B_0$ such that $B_0 \setminus \{ c \}$ is $p$-independent in $Ew$.    
  \end{enumerate}
  By an exhaustive search, we can find a derivation in a proof calculus
  of $\varphi$ from a finite collection $\Phi$ of such axioms.

  Let $v_1, \dotsc, v_n \in \mathbb{P}_{F/K}$ be an enumeration of the finitely many places
  at which either one of the finitely many non-zero constants from $F$ mentioned in $\Phi$
  has non-zero value,
  or which occur as exceptional places in Lemma~\ref{lem:ff-fs-exceptional-places};
  it is clear that such an enumeration can be effectively found
  (noting that the exceptional places in Lemma~\ref{lem:ff-fs-exceptional-places} do not depend on $\Phi$
  and therefore can be fixed in advance).
  Let us consider a place $v \in \mathbb{P}_{F/K}$ not among the $v_i$.
  We claim that all the axioms in $\Phi$, and therefore $\varphi$, are satisfied in the completion $\widehat{F}^v$.
  Certainly $\widehat{F}^v$ is a henselian non-trivially valued field,
  and all the non-zero constants from the axioms $\Phi$ have value zero.
  Further, by choice of $v$ avoiding the exceptional places in Lemma~\ref{lem:ff-fs-exceptional-places},
  the elements $ds \otimes 1 \in \Omega_{\mathcal{O}_v} \otimes Fv$, $s \in B_K \cup \{ t \}$,
  are $Fv$-linearly independent
  and so the formal smoothness axioms are satisfied.
  The residue field $Fv$ is a finite extension of $K$,
  and so it satisfies axiom schemes (1) and (3) from Theorem~\ref{thm:almost-all-res-pac}
  and has the same imperfect exponent as $K$.
  Lastly, there exists a $c \in B_K \cup \{ t + \mathfrak{m}_v\}$
  such that the $ds \in \Omega_{Fv}$, $s \in (B_K \cup \{ t + \mathfrak{m}_v \}) \setminus \{ c \}$,
  are $Fv$-linearly independent,
  and so $(B_K \cup \{ t + \mathfrak{m}_v \}) \setminus \{ c \}$ is $p$-independent in $Fv$.
  Recalling that $t$ is identified with $t + \mathfrak{m}_v$ when we see $Fv$ as an $\mathcal{L}_{\mathrm{ring}}(F)$-structure,
  this shows that $Fv$ satisfies those axioms on residue fields $Ew$
  contained among the $\Phi$.
  This proves our claim that $\widehat{F}^v \models \Phi$ and so $\widehat{F}^v \models \varphi$.

  Therefore it only remains to check whether $\widehat{F}^{v_i} \models \varphi$ for $i = 1, \dotsc, n$.
  This is a decidable problem by Lemma~\ref{lem:single-completion-decidable}.
\end{proof}

\section*{Acknowledgements}

I would like to thank the Hausdorff Research Institute for Mathematics Bonn,
funded by the DFG under Germany's Excellence Strategy (EXC-2047/1 -- 390685813),
for its hospitality during the autumn 2025 trimester programme ``Definability, decidability, and computability''
while parts of this paper were written.

Some of the results of this paper, still in conjectural form, were first presented at the symposion
``À la recherche de l'inconnu: Challenges in the model theory of local and global fields'' in 2024;
I would like to thank its organiser, Jochen Koenigsmann, and the other participants for the ensuing discussion.
I am grateful to Franz-Viktor Kuhlmann for a discussion of references concerning Remark~\ref{rem:lu-discrete}.
I would also like to thank Arno Fehm for pointing me toward \cite{JardenPoonen_galois-points-varieties},
and for numerous helpful comments on a draft version.

\printbibliography

\end{document}